\documentclass[11pt]{amsart}
\usepackage{amsmath} 
\usepackage{amsthm}
\usepackage{amssymb} 
\usepackage{amsfonts}
\usepackage{graphicx} 
\usepackage{color,comment}
\usepackage{soul,xcolor}
\usepackage{ccaption}
\usepackage[shortlabels]{enumitem}
\setlist[enumerate]{label=\normalfont{(\roman*)}}

\newcommand{\sred}[1]{\textcolor{red}{\st{#1}}}

\newtheorem{theorem}{Theorem}[section]
\newtheorem{lemma}[theorem]{Lemma}
\newtheorem{proposition}[theorem]{Proposition}
\newtheorem{corollary}[theorem]{Corollary}
\newtheorem{claim}[theorem]{Claim}

\theoremstyle{definition}

\newtheorem{definition}[theorem]{Definition}
\newtheorem{question}[theorem]{Question}
\newtheorem{remark}[theorem]{Remark}
\newtheorem{exercise}[theorem]{Exercise}

\newtheorem*{notation}{Notation}

\setstcolor{red}

\newcommand{\pt}{\mathrm{pt}}
\newcommand{\Z}{\mathbb{Z}}

\newcommand{\CP}{\mathbb{C}P}
\newcommand{\RP}{\mathbb{R}P}

\let\int\relax
\newcommand{\int}{\mathring}

\usepackage[margin=1.5in]{geometry}

\newcommand{\boundary}{\partial}
\newcommand{\id}{\text{id}}
\newcommand{\into}{\hookrightarrow}
\newcommand{\D}{\mathcal{D}}
\newcommand{\K}{\mathcal{K}}
\newcommand{\CS}{\text{CS}}
\newcommand{\R}{\mathbb{R}}
\let\S\relax
\newcommand{\S}{\mathcal{S}}
\newcommand{\tri}{\mathcal{T}}
\title[Isotopies of surfaces in 4-manifolds via banded unlink diagrams]{Isotopies of surfaces in 4-manifolds via banded unlink diagrams}
\author[Mark Hughes, Seungwon Kim, Maggie Miller]{Mark Hughes, Seungwon Kim, and Maggie Miller}

\begin{document}
\maketitle

\begin{abstract}
In this paper, we study surfaces embedded in $4$-manifolds. We give a complete set of moves relating banded unlink diagrams of isotopic surfaces in an arbitrary $4$-manifold. This extends work of Swenton and Kearton-Kurlin in $S^4$. 
As an application, we show that bridge trisections of isotopic surfaces in a trisected $4$-manifold are related by a sequence of perturbations and deperturbations, affirmatively proving a conjecture of Meier and Zupan. We also exhibit several isotopies of unit surfaces in $\CP^2$ (i.e. spheres in the generating homology class), proving that many explicit unit surfaces are isotopic to the standard $\CP^1$. This strengthens some previously known results about the Gluck twist in $S^4$, related to Kirby problem 4.23.
\end{abstract}

\section{Introduction}\label{sec:introduction}

Knotted surfaces in 4-manifolds play an important role in smooth 4-dimensional topology, analogous to the part played by classical knots in 3-dimensional topology.  Much like in the 3-dimensional case, there are a number of surgery operations and invariants for a smooth 4-manifold $X^4$ which are defined in terms of embedded surfaces inside $X^4$.

Because of their importance in 4-manifold topology, it is useful to have concrete ways of describing these embedded surfaces and their isotopies.   When $X^4 = S^4$, there are several ways to describe embedded surfaces and isotopies between them.  These include broken surface diagrams with Roseman moves~\cite{roseman1998reidemeister}, motion picture presentations with movie moves~\cite{carter1993reidemeister,carter1998knotted}, and braid charts with chart moves~\cite{kamada19932,kamada2002braid}.

In this paper we consider two additional methods for describing surfaces in a $4$-manifold.  When the underlying 4-manifold is $S^4$,  a complete set of moves to describe isotopies of these surfaces has already been established.  We focus here on establishing complete sets of moves to describe surface isotopies in an arbitrary 4-manifold.

The first method we consider to describe a surface $\Sigma$ in a $4$-manifold $X^4$ is via \emph{banded unlink diagrams}. When $X^4=S^4$, this construction involves putting $\Sigma$ into Morse position with respect to a standard height function $h$ on $S^4$, and then encoding the index-0 and index-1 critical points of $h\vert_{\Sigma}$ as a classical unlink in $S^3$ with a collection of embedded bands attached (see Section~\ref{sec:banddiagrams} for a more detailed description).  In~\cite{yoshikawa1994enumeration}, Yoshikawa presents a set of moves on banded unlink diagrams for surfaces in $S^4$ which are realizable by isotopies of the underlying surface, and asks if these moves are sufficient to relate banded unlink diagrams of any pair of isotopic surfaces.  This question was affirmatively answered by Swenton~\cite{swenton2001calculus}, with an alternate proof being given by Kearton and Kurlin~\cite{kearton2008all}.

In this paper we study a generalization of banded unlink diagrams to embedded surfaces in an arbitrary 4-manifold $X^4$ equipped with a Morse function, where we encode the Morse function by a Kirby diagram $\mathcal{K}$.  We describe a set of moves on banded unlinks, called \emph{band moves}, which can be realized by isotopies of the underlying surface $\Sigma$.  These consist of Yoshikawa's original moves, as well as additional moves which describe the interaction of the surface $\Sigma$ with the handle structure on $X^4$.  The main theorem we prove is the following.
\theoremstyle{theorem}
\newtheorem*{thm:generalswenton}{Theorem~\ref{thm:generalswenton}}
\begin{thm:generalswenton}
Let $X^4$ be a smooth 4-manifold with Kirby diagram $\mathcal{K}$, and suppose that $\Sigma$ and $\Sigma'$ are embedded surfaces in $X^4$.  Let $(\mathcal{K},L,v)$ and $(\mathcal{K},L',v')$ be banded unlink diagrams for $\Sigma$ and $\Sigma'$ respectively.  Then $\Sigma$ and $\Sigma'$ are isotopic if and only if $(\K,L,v)$ can be transformed into $(\K,L',v')$ by a finite sequence of band moves.
\end{thm:generalswenton}

The second method we consider to represent an embedded surface $\Sigma$ is by using a \emph{bridge trisection} of $\Sigma$, which allows one to present $\Sigma$ in terms of intersections with a given trisection of the ambient manifold $X^4$.  Bridge trisections for surfaces in $S^4$ were introduced by Meier and Zupan in~\cite{meier2017bridge}, where they provide a stabilization/destabilization move which they prove is sufficient to relate any two bridge trisections of isotopic surfaces.  In~\cite{meier2018bridge} the same authors generalize this notion of bridge trisections to surfaces in an arbitrary 4-manifold $X^4$, and prove that every surface $\Sigma \subset X^4$ can be put into bridge trisected position with respect to any given trisection on $X^4$.  They similarly define a stabilization/destabilization move 
and conjecture that these moves are sufficient to relate any two bridge trisections of  isotopic surfaces in $X^4$. 
Using Theorem~\ref{thm:generalswenton}, we affirmatively answer this conjecture.  We give the relevant definitions and some exposition on trisections and bridge trisections in Section~\ref{sec:bridge}.


\newtheorem*{bridgetheorem}{Theorem~\ref{bridgethm}}
\begin{bridgetheorem}
Let $S$ and $S'$ be surfaces in bridge position with respect to a trisection $\tri$ of a closed $4$-manifold $X^4$. Suppose that $S$ is isotopic to $S'$. Then $S$ can be taken to $S'$ by a sequence of perturbations and deperturbations, followed by a $\tri$-regular isotopy.
\end{bridgetheorem}

As a separate application of Theorem~\ref{thm:generalswenton}, we focus on the case of unit surfaces in $\CP^2$. By Melvin~\cite{melvin1977thesis}, the study of unit surfaces is relevant to understanding the Gluck twist surgery of~\cite{gluck1961embedding}.  Melvin showed that the Gluck twist on a sphere $S\subset S^4$ yields $S^4$ again if and only if there is a diffeomorphism from the pair $(\CP^2,S\#\CP^1)$ to the pair $(\CP^2,\CP^1)$. See Section~\ref{sec:cp2}, where we give the relevant definitions and exposition, for more detail.

\newtheorem*{cp2theorem}{Theorem~\ref{cp2theorem}}
\begin{cp2theorem}
Let $F=S\#\CP^1\subset\CP^2$ be a genus-$g$ unit surface knot, where $S\subset S^4$ is an orientable surface that is $0$-concordant to a band-sum of twist-spun knots and unknotted surfaces. Then $F$ is isotopic to $\CP^1\#gT$, where $\CP^1\# gT$ indicates the standard $\CP^1$ trivially stabilized $g$ times.
\end{cp2theorem}

\pagebreak

\subsection*{Outline}\begin{itemize}
\item[] {\bf{Section~\ref{sec:levelsetsandisotopies} (Pages~\pageref{sec:levelsetsandisotopies} --~\pageref{endsec1}):}} We define horizontal-vertical position and some nice families of isotopies for surfaces in $4$-manifolds.
\item[] {\bf{Section~\ref{sec:banddiagrams} (Pages~\pageref{sec:banddiagrams} --~\pageref{endsec2}):}} We define banded unlink diagrams.
\item[] {\bf{Section~\ref{sec:swenton} (Pages~\pageref{sec:swenton} --~\pageref{endsec3}):}} We show that any surface in a $4$-manifold is described by a banded unlink diagram which is well-defined up to a certain set of moves.
\item[] {\bf{Section~\ref{sec:bridge} (Pages~\pageref{sec:bridge} --~\pageref{endsec4}):}} We show that a bridge trisection of a surface in a trisected $4$-manifold is unique up to perturbation.
\item[] {\bf{Section~\ref{sec:cp2} (Pages~\pageref{sec:cp2} --~\pageref{endsec5}):}} We consider many examples of surfaces in the generating homology class of $\CP^2$, and show explicitly that these examples are isotopic to $\CP^1$ (perhaps with trivial tubes attached if the original surface is of positive genus).
\end{itemize}

\subsection*{Remark} In an earlier version of this paper we claimed to have shown that all unit knots in $\CP^2$ were in fact smoothly isotopic to $\CP^1$, and hence that all Gluck twists on $S^4$ were standard (Theorems~1.4 and 1.2 respectively), though we discovered a gap in our proof of Proposition~3.1. In this earlier version, we used diagrammatic moves of banded unlinks in $\CP^2$ and applied  techniques from a recent draft by Kawauchi. In this present version we focus on these diagrammatic moves, fleshing out a complete set of moves relating diagrams of isotopic surfaces in a general $4$-manifold. The previous work in $\CP^2$ is now contained in Section~\ref{sec:cp2}.

\subsection*{Acknowledgements}
The authors would like to thank Alexander Zupan for interesting conversations on possible unknotting strategies in $\CP^2$.
The second author is supported by the National Institute for Mathematical Sciences (NIMS).
The third author is a fellow in the National Science Foundation Graduate Research Fellowship program,
under Grant No. DGE-1656466. She would like to thank her graduate advisor, David Gabai.

\section{Level sets and isotopies}\label{sec:levelsetsandisotopies}


\subsection{Banded links}  Our primary technique for studying embedded surfaces in a $4$-manifold will be to arrange them so that their intersections with the level sets of a given Morse function are composed of disjoint unions of embedded disks and \emph{banded links}.

More precisely, let $M^3$ be an oriented 3-manifold, and let $L \subset M$ be a link.  A \emph{band} $b$ for the link $L$ is the image of an embedding $\phi : I \times I \hookrightarrow M$, where $I = [-1,1]$, and $b \cap L = \phi (\{ -1,1\} \times I)$.  We call $\phi(I \times \{\tfrac{1}{2}\})$ the \emph{core} of the band $b$.  Let $L_b$  be the link defined by 
\[
L_b = (L \backslash \phi (\{-1,1\} \times I)) \cup \phi (I \times \{-1,1\}).
\]
Then we say that $L_b$ is the result of performing \emph{band surgery} to $L$ along $b$.  If $v$ is a finite family of pairwise disjoint bands for $L$, then we will let $L_v$ denote the link we obtain by performing band surgery along each of the bands in $v$.  We say that $L_v$ is the result of \emph{resolving} the bands in $v$.  The union of a link $L$ and a family of disjoint bands for $L$ is called a \emph{banded link}.   If $L$ is an unlink, we call the union of $L$ and a family of disjoint bands a \emph{banded unlink}.

\subsection{Horizontal and vertical sets} 
Now let $X^4$ denote a closed, oriented, $4$-manifold equipped with a self-indexing Morse function $h:X^4\to[0,4]$, where $h$ has exactly one index-$0$ critical point and one index-$4$ critical point. We will write $\K$ to denote the Kirby diagram of $X^4$ induced by $h$ (we explain this more precisely in Subsection~\ref{sec:diagrams}). 

In order to study $\Sigma \subset X^4$ via the level sets of $h$, it will be convenient to have a way of identifying subsets of distinct level sets $h^{-1}(t_1)$ and $h^{-1}(t_2)$.  Suppose then that $t_1\leq t_2$, and let $x_1, \ldots, x_p$ denote the critical points of $h$ which satisfy $t_1\leq h(x_j) \leq t_2$.  Let $X_{t_1,t_2}$ denote the complement in $X^4$ of the ascending and descending manifolds of the critical points $x_1, \ldots, x_p$.  Then the gradient flow of $h$ defines a diffeomorphism $\rho_{t_1,t_2}:h^{-1}(t_1) \cap X_{t_1,t_2} \rightarrow h^{-1}(t_2) \cap X_{t_1,t_2}$.

\begin{definition}
We call $\rho_{t_1,t_2}$ the \emph{projection of $h^{-1}(t_1)$ to $h^{-1}(t_2)$}.  Similarly, we call $\rho^{-1}_{t_1,t_2}$ the \emph{projection of $h^{-1}(t_2)$ to $h^{-1}(t_1)$}, which we likewise denote by $\rho_{t_2,t_1}$.  
 \end{definition}
 
Note that despite calling $\rho_{t_1,t_2}$ the projection from $h^{-1}(t_1)$ to $h^{-1}(t_2)$, it is only defined on the complement of the ascending and descending manifolds of the critical points that sit between $t_1$ and $t_2$.  These projection maps allow us to define local product structures away from the ascending and descending manifolds of the critical points of $h$.

\begin{definition}Let $W$ be a subset of $X^4$, and let $J$ either be the closed interval $[t_1,t_2]$ or the open interval $(t_1,t_2)$.  Then we say that $W$ is \emph{vertical on the interval $J$} if $W \subset X_{t_1,t_2}$ and if $\rho_{t,t'} (h^{-1}(t) \cap W)=h^{-1}(t') \cap W$ for all $t,t' \in J$.  
 \end{definition}

In Sections~\ref{sec:banddiagrams} and~\ref{sec:swenton}, we will construct isotopies of surfaces in $X^4$. In this paper, every isotopy of a surface will always extend to ambient isotopy. 
We generally write ``$f:F\times I\to X^4\times I$ is an isotopy'' with the understanding that there is in fact a smooth family of diffeomorphisms $g_s:X^4\to X^4$ with $g_0=\id$ so that $g_s\circ \operatorname{pr}_{X^4} \circ f(F\times 0)=f(F\times s)$ for $s\in[0,1]$.  Here we use $\operatorname{pr}_X:X^4 \times I \rightarrow X^4$ to denote projection to the first factor.

We consider a few special types of isotopy which behave well with respect to $h$.

\begin{definition}
Let $\Sigma\subset X^4$ be a smoothly embedded surface. Let $f: F\times I\to X^4\times I$ be a smooth isotopy of $\Sigma$, so that $\Sigma=f( F\times \left\{0\right\})$.   If the image of $\operatorname{pr}_X \circ f$ is disjoint from the critical points of $h$, then we say that $f$ is {\emph{$h$-disjoint}}.  

We say that $f$ is {\emph{horizontal}} with respect to $h$ if $h(\operatorname{pr}_X(f(x,s)))$ is independent of $s$ for all $x\in F$. We say that $f$ is {\emph{vertical}} with respect to $h$ if for each $x \in F$ the image of $\left\{x\right\} \times I$ under $\operatorname{pr}_X \circ f$ is contained in a single orbit of the flow of $\nabla h$.  Finally, we say that $f$ is \emph{$h$-regular} if for each $s\in I$, $h\vert_{\operatorname{pr}_X(F\times \{s\})}$ is Morse.  
(See Figure~\ref{fig:hzisotopies} for schematics of horizontal and vertical isotopies.)
\end{definition}

\begin{figure}
\includegraphics[width=\textwidth]{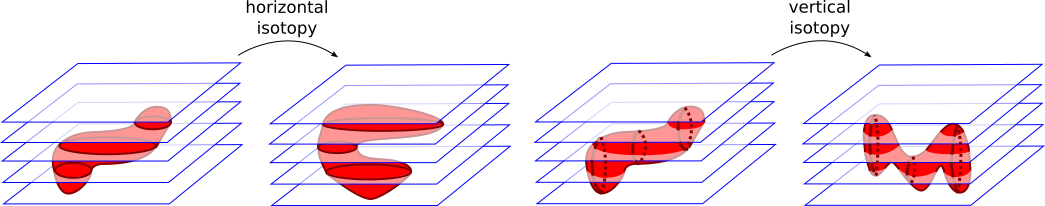}
\caption{The horizontal planes represent $h^{-1}(t)$ for various values of $t$. {\bf{Left:}} A horizontal isotopy of $\Sigma$ preserves $h\vert_\Sigma$ pointwise. {\bf{Right:}} A vertical isotopy of $\Sigma$ moves each $x\in\Sigma$ within a single orbit flow of $\nabla h$.}\label{fig:hzisotopies}
\end{figure}

Intuitively, one should think of $h$ as a height function whose level sets are horizontal. A horizontal isotopy of $\Sigma$ moves $\Sigma\cap h^{-1}(t)$ within $h^{-1}(t)$, preserving $h\vert_\Sigma$. A vertical isotopy of $\Sigma$ changes $h\vert_\Sigma$, but preserves the projection of $\Sigma$ onto each level set $h^{-1}(t)$.  We will usually just say that $f$ is horizontal or vertical, omitting ``with respect to $h$".

Note that a surface isotopy is generically $h$-disjoint. We will often explicitly require isotopies to be $h$-disjoint because we will be interested in how isotopy affects the projection of surfaces to $h^{-1}(3/2)$, which can be complicated when the isotopy takes the surface through a critical point. The definitions of horizontal and vertical isotopy naturally motivate a ``nice'' position of a surface embedded in $X^4$, if we are willing to allow the surface embedding to have corners.

\begin{definition}\label{hzdef}
Let $\Sigma\subset X^4$ be a PL embedded surface. We say that $\Sigma$ is in {\emph{horizontal-vertical position}} with respect to $h$ if there exists a set $T=\{t_1,\ldots,t_n\}$ disjoint from $\{0,1,2,3,4\}$, with $t_1<t_2<\cdots< t_n$, so that the following are true: 
\begin{itemize}
\item For each $1\leq i \leq n-1$ the surface $\Sigma$ is vertical on the interval $(t_i,t_{i+1})$.
\item For each $1\leq j \leq n$ the intersection $\Sigma\cap h^{-1}(t_j)$ consists of the disjoint union of a (possibly empty) banded link and a  (possibly empty) union of disjoint embedded disks.
\end{itemize}
In other words, $\Sigma$ is vertical away from a finite number of nonsingular level sets $h^{-1}(t_i)$, while it intersects the $h^{-1}(t_i)$ in a collection of horizontal disks and a banded link.  When $h$ is clear, we may simply say that $\Sigma$ is \emph{horizontal-vertical}.
\end{definition}

Note that from the definition of vertical, a horizontal-vertical surface must be disjoint from critical points of $h$. Note furthermore, that by an arbitrarily small perturbation in a neighborhood of the level sets $h^{-1}(t_1), \ldots , h^{-1}(t_n)$, a horizontal-vertical surface $\Sigma$ may be isotoped to a surface $\Sigma'$ with $h|_{\Sigma'}$ Morse.  This perturbation can be chosen so that each horizontal band of $\Sigma$ gives rise to a nondegenerate saddle point in $\Sigma'$, and each horizontal disk in $\Sigma$ gives rise to a nondegenerate maximum or minimum point in $\Sigma'$.  We will thus work largely with nonsmooth surfaces that are horizontal-vertical when constructing isotopies, with the understanding that they may be isotoped into smooth surfaces in Morse position as described above.


The following  classical theorem states that arbitrary surfaces can be put into horizontal-vertical position. This is critical to the study of surfaces embedded in $4$-manifolds. A proof for orientable surfaces is essentially contained in Section 2 of~\cite{kawauchi1982normalform}; the nonorientable case is covered in~\cite{kamada1989nonorientable}.

\begin{theorem}[\cite{kamada1989nonorientable,kawauchi1982normalform}]\label{normalform}
Let $\Sigma\subset X^4$ be a smoothly embedded surface so that $h\vert_\Sigma$ is Morse and $\Sigma$ is disjoint from the critical points of $h$. Then there is an $h$-disjoint and $h$-regular isotopy $f: F\times I\to X^4\times I$ with $f( F\times \{0\})=\Sigma$ so that $f$ is a concatenation of horizontal and vertical isotopies, $f( F\times \{1\})$ is horizontal-vertical, and $h\vert_{f( F\times \{s\})}$ is Morse for all $s\in[0,1)$.
\end{theorem}

Both~\cite{kamada1989nonorientable} and~\cite{kawauchi1982normalform} consider only surfaces embedded in $S^4$ with the standard height function, but by applying the argument locally the theorem can be extended to surfaces in an arbitrary $4$-manifold $X^4$ with self-indexing Morse function $h$. We will not cite this theorem directly, but will implicitly prove this result in Lemma~\ref{generictoband}.
\label{endsec1}

\section{Banded unlink diagrams for surfaces in $4$-manifolds}\label{sec:banddiagrams}

\subsection{Banded unlink position}
For ease of notation we let $M_t = h^{-1}(t)$ denote the (possibly singular) 3-dimensional level set at height $t$ for each $t \in \mathbb{R}$.  By the projection maps $\rho_{t_1,t_2}$ we may identify subsets of distinct level sets $M_t$, provided they avoid the appropriate ascending and descending manifolds.  We will often make these identifications implicitly, so for example, we may think of a link $L$ as living in both $M_{t_1}$ and $M_{t_2}$ when there is no risk of confusion. 

\begin{definition}
We say that an embedded surface $\Sigma \subset X^4$ is in \emph{banded unlink position} if 
\begin{itemize}
\item $h(\Sigma) = [1/2,5/2]$,
\item $\Sigma$ is vertical on the intervals $(1/2,3/2)$ and $(3/2,5/2)$,
\item $\Sigma \cap M_{3/2}$ is a banded unlink disjoint from the descending manifolds of index-$2$ critical points of $h$, and 
\item $\Sigma \cap M_{1/2}$ and $\Sigma \cap M_{5/2}$ are finite collections of disjoint embedded disks.
\end{itemize}
\end{definition}

Letting $t$ denote the height coordinate on $X^4$ induced by $h$, we can describe a surface in banded unlink position by a movie in $t$ as follows.  Starting at $t=0$ and increasing, we first encounter a collection of minimal disks of $\Sigma$ at height $t=1/2$.   For $t\in(1/2,3/2)$, the intersection $\Sigma\cap M_t$ is an unlink in $M_t$, which we denote by $L$.  
The next feature we encounter are the index 1 critical points of $X^4$ at height $t=1$, which completes the 1-skeleton of $X^4$. 
As we continue upwards, at height $t=3/2$ a family $v$ of bands appear, attached to the link $L$.  Passing $t=3/2$, the resulting level set of $\Sigma$ becomes $L_v$, which is obtained from $L$ by resolving the bands $v$.  We then pass the index 2 critical points of $X$ at height $t=2$, before finally capping off the components of $L_v$ with maximal disks at height $t=5/2$.

Note that the link $L$ is necessarily an unlink in $M_t$ for $t \in (1/2,3/2)$ (i.e.\ it bounds a collection of disjoint embedded disks), and $L_v$ will be an unlink in $M_t$ for $t \in (2,5/2)$.

\subsection{Banded unlink diagrams}\label{sec:diagrams}

Surfaces in banded unlink position can be represented in terms of the associated Kirby diagram of $X^4$ via \emph{banded unlink diagrams}.  Suppose that the handle decomposition induced by $h$ on $X^4$ is represented by the Kirby diagram $\mathcal{K} \subset S^3$.  More precisely, $\mathcal{K}$ is a link $L_1 \sqcup L_2 \subset S^3$, where $L_1$ is an unlink, and each component of $L_2$ is labeled with an integer framing.  The components of $L_1$ are each decorated with a dot to distinguish them from the components of $L_2$, and each indicates a 1-handle attached to the 0-handle $B^4$ along $\partial B^4 = S^3$ as usual (the meridians of $L_1$ are cores of the $1$-handles in the handle decomposition of $X^4$).  The labeled components of $L_2$ each represent the framed attaching circle of a 2-handle attached to $X_1$.

Given such a Kirby diagram $\mathcal{K} \subset S^3$ for $X^4$, the sphere $S^3$ can be identified with $M_{1/2}$, while the 3-manifold obtained by performing 0-surgery to $S^3$ along $L_1$ can be identified with the level set $M_{3/2}$.  
After performing this surgery $L_2$ can again be thought of as a framed link in $M_{3/2}$, and we identify the result of performing Dehn surgery to $M_{3/2}$ along the components of $L_2$ (where the surgery coefficient of each component is specified by its framing) with $M_{5/2}$.  
Let $E(\K)$ denote the complement $S^3 \backslash (\nu(L_1) \cup \nu(L_2))$ of a small tubular neighborhood of $\mathcal{K} = L_1 \sqcup L_2$ in $S^3$.  Then given a link $L \subset E(\K)$ we may think of $L$ as describing links in $M_{1/2}$,  $M_{3/2}$ and $M_{5/2}$ in the obvious way. 

A \emph{banded unlink diagram} in the Kirby diagram $\mathcal{K}$ is a triple $(\mathcal{K}, L, v)$, where $L \subset E(\K)$ is a link and $v$ is a finite family of {\emph{disjoint}} bands for $L$ in $E(\K)$, such that $L$ bounds a family of pairwise disjoint embedded disks in $M_{1/2}$, and $L_v$ bounds a family of pairwise disjoint embedded disks in $M_{5/2}$. See Figure~\ref{fig:bandedunlink} for an example of a banded unlink diagram.

\begin{figure}
\includegraphics[width=.3\textwidth]{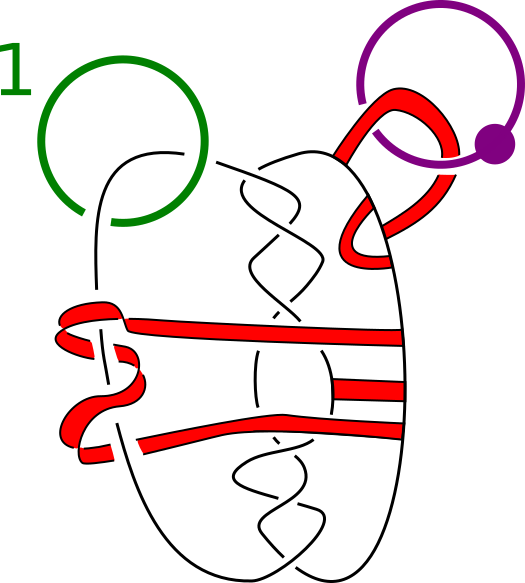}
\caption{A banded unlink in Kirby diagram $\K$ describing a torus $\Sigma$  smoothly embedded in $\CP^2\#(S^1\times S^3)$. The 2-component unlink in $E(\K)$ bounds two minima of $\Sigma$. Resolving the unlink along the four bands yields the boundary of two maxima of $\Sigma$. Then $\chi(\Sigma)=2-4+2=0$. One can check also that $\Sigma$ is orientable.}\label{fig:bandedunlink}
\end{figure}


A banded unlink diagram describes an embedded surface $\Sigma$ in banded unlink position as follows.  We first note that we can identify $E(\K)$ with a subset of $M_{3/2}$ in a natural way.  When fixing this identification, note that the intersection of $M_{3/2}$ with the descending manifolds (cores) of the 2-handles of $X^4$ can be thought of as the attaching circles of the 2-handles.  Hence, as our banded link $L\cup v$ sits in the complement of a tubular neighborhood of the attaching circles $L_2 \subset S^3$, it can be identified with a subset of $M_{3/2}$, which we denote by $L'\cup v'$, that misses the descending manifolds of the 2-handles of $X^4$.  

Now, as $L'$ is an unlink, we can apply a horizontal isotopy in $M_{3/2}$ if necessary so that $L'$ also avoids the ascending manifolds of the 1-handles of $X^4$.  We can thus extend $L'$ vertically downwards from $M_{3/2}$ to $M_{1/2}$, where it can be capped off by a family of disjoint embedded disks in $M_{1/2}$.  Similarly, we can extend the surgered link $L'_{v'}$ vertically upwards from $M_{3/2}$ to $M_{5/2}$, where it can be capped off. 
As these families of disks are unique up to isotopy rel boundary, the surface we obtain in this way from the banded unlink diagram $(\mathcal{K}, L, v)$ is well-defined up to isotopy.  We denote this surface by $\Sigma(\mathcal{K},L,v)$. 
We say that $(\K,L,v)$ {\emph{describes}} $\Sigma(\K,L,v)$, or that $(\K,L,v)$ is a {\emph{banded unlink diagram for}} $\Sigma(\K,L,v)$.


\subsection{Band moves}\label{sec:diagramforbandposition} We now proceed to describe a collection of moves which will allow us in Section~\ref{sec:swenton} to define banded unlink diagrams of arbitrary surfaces in $X^4$, and relate the banded unlink diagrams of any isotopic surfaces.  These moves are described in Figures~\ref{fig:oldbandmoves} and~\ref{fig:newbandmoves}.  They consist of \emph{cup} and \emph{cap} moves (top of Fig.~\ref{fig:oldbandmoves}), \emph{band slides} (middle of Fig.~\ref{fig:oldbandmoves}), \emph{band swims} (bottom of Fig.~\ref{fig:oldbandmoves}), \emph{2-handle-band slides} (top of Fig.~\ref{fig:newbandmoves}), \emph{dotted circle slides} (middle two rows of Fig.~\ref{fig:newbandmoves}) and \emph{2-handle-band swims} (bottom of Fig.~\ref{fig:newbandmoves}). 
These operations, together with isotopy in $E(\K)$, form a collection of moves which we refer to as \emph{band moves}. (Note that the dotted circle slide may actually move $L$ rather than a band, but we still refer to this as a band move for convenience.)  Band moves may transform a banded unlink diagram $(\mathcal{K},L,v)$ into a banded unlink diagram $(\K,L',v')$, though it is not difficult to verify that the surfaces $\Sigma(\K,L,v)$ and $\Sigma(\K,L',v')$ are isotopic.

\begin{figure}\includegraphics[width=.4\textwidth]{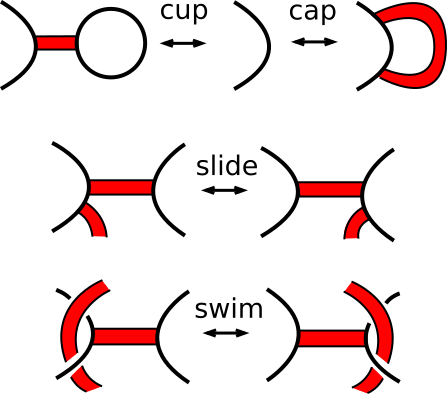}
\caption{The cup/cap, slide, and swim band moves. These band moves do not involve the $2$-handle attaching circles of $\K$. The cup/cap moves correspond to $0$-,$1$- or $2$-,$3$- stabilization/destabilization of a surface $\Sigma$ with respect to $h$. The slide move passes an end of one band along the length of a distinct band. The swim move passes a band lengthwise through the interior of a distinct band.}\label{fig:oldbandmoves}
\end{figure}

\begin{figure}\includegraphics[width=.7\textwidth]{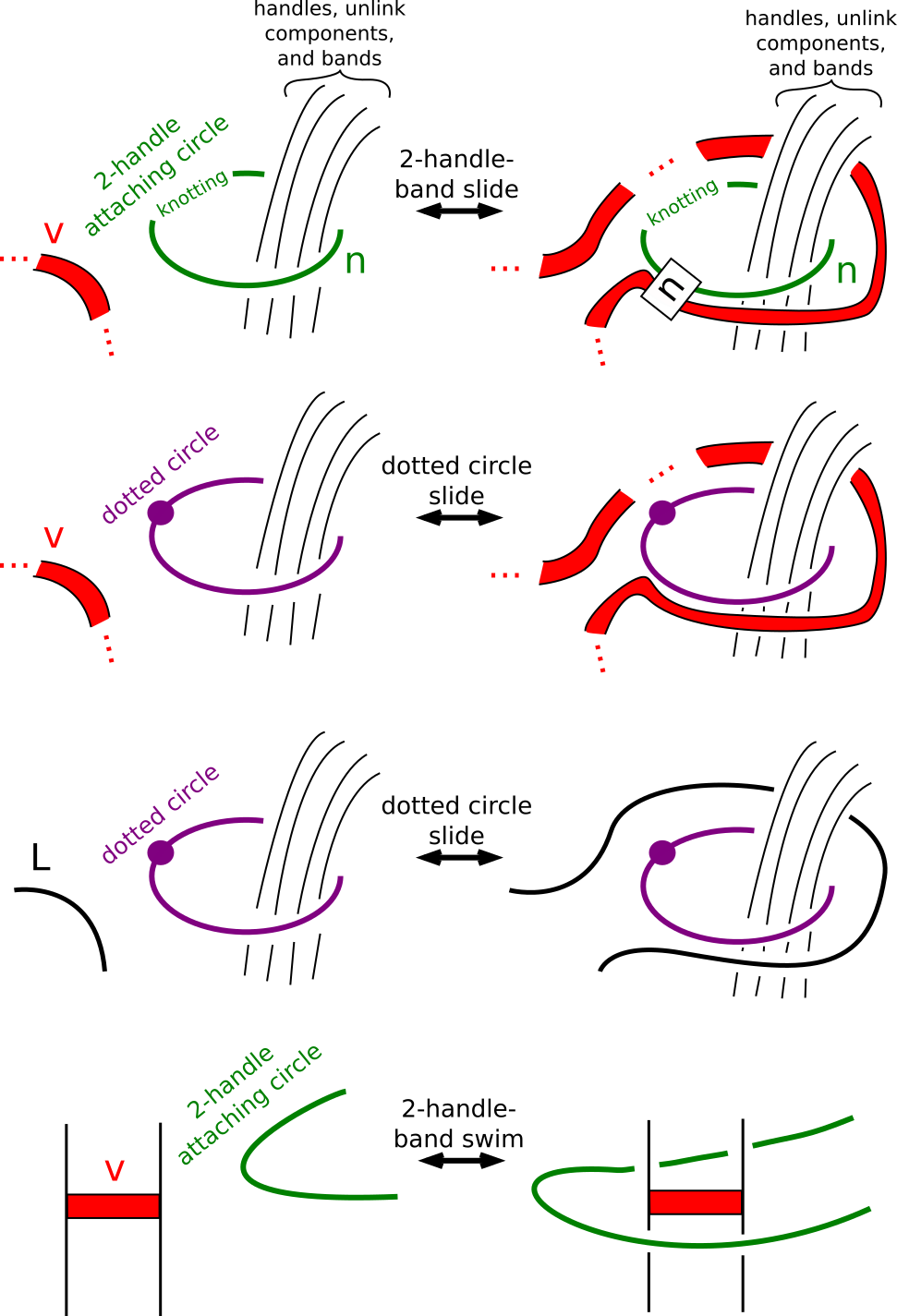}
\caption{The $2$-handle-band slide and $2$-handle-band swim moves. These band moves involve the $2$-handles attaching circles of $\K$. {\bf{Top:}} The $2$-handle-band slide move slides a band over a $2$-handle, following the usual rules of Kirby calculus. This schematic is meant to indicate that the $2$-handle attaching circle may be knotted and link arbitrarily with other circles in $\K$ or unlink or band components (including the band that slides). {\bf{Second row:}} A dotted circle slide may pass a band over a dotted circle, following the usual rules of Kirby calculus. {\bf{Third row:}} A dotted circle slide may pass the unlink $L$ over a dotted circle, following the usual rules of Kirby calculus. {\bf{Bottom:}} The $2$-handle-band swim move passes a $2$-handle attaching circle lengthwise through the interior of a band.}\label{fig:newbandmoves}
\end{figure}

\begin{notation}
We refer to the  slide, swim, 2-handle-band slide, dotted circle slide, 2-handle-band swim moves and isotopy in $E(\K)$ as {\emph{Morse-preserving band moves}}. The cup and cap moves are not Morse-preserving.
\end{notation}

\begin{lemma}\label{diagramfromband}
Let $\Sigma\subset X^4$ be a surface in banded unlink position. There is a procedure to obtain a banded unlink diagram $(\K,L_\Sigma,v_\Sigma)$ so that $L_\Sigma$ and $v_\Sigma$ are completely determined by the embedding $\Sigma\into X^4$. Moreover, $\Sigma$ is isotopic to $\Sigma':=\Sigma(\K,L_\Sigma,v_\Sigma)$.
\end{lemma}
\begin{proof}


Let $L_\Sigma=\Sigma\cap M_{3/2-\epsilon}$ and let $v_\Sigma$ be the bands of $\overline{(\Sigma\cap M_{3/2})\setminus L_{\Sigma}}$.
By definition of banded unlink position, $L_\Sigma$ bounds a system of disks in $M_{1/2}$ (e.g. $\Sigma\cap M_{1/2}$) and ${L_{\Sigma}}_{v_\Sigma}$ bounds a system of disks in $M_{5/2}$ (e.g. $\Sigma\cap M_{5/2}$). Since $\Sigma$ is in banded unlink position, $L_{\Sigma}\cup v_\Sigma$ is contained in $E(\K)$. Therefore, $(\K,L_\Sigma,v_\Sigma)$ is a banded unlink diagram.


Let $\Sigma':=\Sigma(\K,L_\Sigma,v_\Sigma)$. Then $\Sigma'$ is also in banded unlink position, with $\Sigma'\cap(1/2,5/2)=\Sigma\cap(1/2,5/2)$. 
Both $\Sigma\cap h^{-1}([0,1/2])$ and $\Sigma'\cap h^{-1}([0,1/2])$ are boundary-parallel disk systems with equal boundary in $h^{-1}([0,1/2])\cong B^4$, so are isotopic rel boundary. Similarly, $\Sigma\cap h^{-1}([5/2,4])$ and $\Sigma'\cap h^{-1}([5/2,4])$ are isotopic rel boundary in $ h^{-1}([5/2,4])\cong\natural(S^1\times B^3)$. Therefore, $\Sigma$ is isotopic to $\Sigma'$.
\end{proof}

\begin{remark}
In Lemma~\ref{diagramfromband}, we showed that if $\Sigma$ and $\Sigma'$ are surfaces  in banded unlink position in $X^4$, then the banded unlink diagrams $\D:=(\K,L_\Sigma, v_\Sigma)$ and $\D':=(\K,L_{\Sigma'},v_{\Sigma'})$ are well-defined. 
However, even if $\Sigma$ and $\Sigma'$ are isotopic, we have not shown that $\D$ and $\D'$ are related in any way.
\end{remark}
\label{endsec2}

\section{A calculus of moves on banded unlink diagrams}\label{sec:swenton}

%
%

\subsection{Overview}\label{sec:overview}

In what follows, let $\K_0$ denote the standard (empty) Kirby diagram induced by the standard height function on $S^4$. (The handle decomposition described by $\K_0$ has one $0$-handle, one $4$-handle, and no other handles).  When $\K=\K_0$, Swenton~\cite{swenton2001calculus} and Kearton-Kurlin~\cite{kearton2008all} show that the cup, cap, band slide, and band swim moves relate any two banded unlink diagrams of isotopic surfaces.

\begin{theorem}[\cite{swenton2001calculus,kearton2008all}\label{thm:swenton}]
Let $\Sigma$ and $\Sigma'\subset S^4$ be isotopic surfaces described by banded unlink diagrams $\D:=(\K_0,L,v)$ and $\D':=(\K_0,L',v')$ respectively. Then $\D'$ can be obtained from $\D$ by a finite sequence of cap/cup, band slides, band swims, and isotopies in $S^3$.
\end{theorem}
Note that we have not defined what it means for an arbitrary surface in $S^4$ to be described by a banded unlink diagram (even with $h$ the standard height function). Part of the content of Theorem~\ref{thm:swenton} is that such a diagram is well-defined.
\begin{definition}
Let $\Sigma$ be a surface in $S^4$. Let $\Sigma'$ be a surface in banded unlink position (with respect to the standard height function) which is isotopic to $\Sigma$. We say that $(\K_0,L_{\Sigma'},v_{\Sigma'})$ is a banded unlink diagram for $\Sigma$. This diagram is well-defined up to cap/cup, band slides, band swims, and isotopy in $S^3$~\cite{swenton2001calculus,kearton2008all}.
\end{definition}

By including 2-handle-band swims, 2-handle-band slides, and dotted circle slides along with the moves in Theorem~\ref{thm:swenton}, we can generalize Theorem~\ref{thm:swenton} to surfaces in arbitrary closed 4-manifolds. We state the theorem now, even though we have not defined what it means for an arbitrary surface in $X^4$ to be described by a banded unlink diagram.

 \begin{theorem}\label{thm:generalswenton}
Let $\Sigma$ and $\Sigma'$ be surfaces $X^4$, with banded unlink diagrams $\D:=(\mathcal{K},L,v)$ and $\D':=(\mathcal{K},L',v')$ respectively.  Then $\Sigma$ and $\Sigma'$ are isotopic if and only if $\D$ can be transformed into $\D'$ by a finite sequence of band moves. 
\end{theorem}

Note that when $X^4 = S^4$ and $\K=\K_0$, Theorem~\ref{thm:generalswenton} reduces to the statement of Theorem~\ref{thm:swenton}. Loosely, to prove Theorem~\ref{thm:generalswenton}, we will analyze how banded unlink diagrams for $\Sigma$ change under isotopy of $\Sigma$. Here is a brief outline of our strategy for proving Theorem~\ref{thm:generalswenton}.

\begin{enumerate}[1.]
\item We show that surfaces in banded unlink position admit banded unlink diagrams well-defined up to band moves. (Section~\ref{sec:diagramforbandposition})
\item We show that surfaces in horizontal-vertical position (Definition~\ref{hzdef}) admit banded unlink diagrams well-defined up to band moves (Section~\ref{sec:diagramforhzposition})
\item We show that certain isotopies of a horizontal-vertical surface preserves the associated banded unlink diagram up to band moves. (Section~\ref{sec:diagramforhzposition})
\item We show that surfaces in the more general {\emph{generic}} position (Definition~\ref{def:singular}) admit banded unlink diagrams well-defined up to band moves. (Section~\ref{sec:diagramforgeneric})
\item\label{step5} We show that certain isotopies between generic surfaces preserve the associated banded unlink diagrams up to band moves. (Section~\ref{sec:arbitrary})
\item We show that any isotopy of surfaces can be perturbed to an isotopy as in Step~\ref{step5} We define the banded unlink diagram of a surface $\Sigma$ to be the banded unlink diagram of any generic surface isotopic to $\Sigma$. (Section~\ref{sec:arbitrary})
\end{enumerate}

\subsection{Banded unlink diagrams for horizontal-vertical surfaces}\label{sec:diagramforhzposition}

We now extend Lemma~\ref{diagramfromband} to horizontal-vertical surfaces, rather than only surfaces in banded unlink position.

\begin{proposition}\label{horizontaltoband}
Let $\Sigma\subset X^4$ be a surface in horizontal-vertical position so that all minima of $h\vert_\Sigma$ are below all saddles of $h\vert_\Sigma$, which are below all maxima of $h\vert_\Sigma$. Then we may obtain a banded unlink diagram $\D=(\K,L_\Sigma,v_\Sigma)$ so that $L_\Sigma$ and $v_\Sigma$ are determined up to Morse-preserving band moves by the embedding of $\Sigma$ into $X^4$. Moreover, $\Sigma$ is isotopic to $\Sigma(\K,L_\Sigma,v_\Sigma)$.
\end{proposition}
\begin{proof}
We will isotope $\Sigma$ into banded unlink position and apply Lemma~\ref{diagramfromband}. 
If necessary, apply a small horizontal isotopy to the horizontal parts of $\Sigma\cap h^{-1}(3-\epsilon,4]$ to avoid intersections with the ascending manifolds of index-$3$ critical points of $h$. Then isotope $\Sigma\cap h^{-1}(5/2-\epsilon,4]$ vertically into $h^{-1}(5/2-\epsilon,5/2)$. 

Next, apply a small horizontal isotopy to the horizontal parts of $\Sigma\cap h^{-1}[0,3/2+\epsilon)$ to avoid intersections with the descending manifolds of index-$1$ critical points of $h$. Then isotope $\Sigma\cap h^{-1}[0,3/2+\epsilon)$ vertically into $h^{-1}(3/2,3/2+\epsilon)$. 

Now isotope horizontal neighborhoods of the minima and maxima of $h\vert_{\Sigma}$ horizontally to avoid the ascending and descending manifolds of index-$2$ critical points of $h$. Isotope the minima vertically to $h^{-1}(1/2)$ and the maxima vertically to $h^{-1}(5/2)$ (apply further horizontal isotopy to the minima and maxima as necessary to ensure no self-intersections are introduced to $\Sigma$.) 

Let $b$ be the horizontal neighborhood of an index-$1$ critical point of $h\vert_\Sigma$ (so $b$ is a band). If $h(b)> 2$ and $b$ intersects the ascending manifold of some index-$2$ critical point of $h$, then isotope $b$ slightly horizontally to either side of the ascending manifold. (These two choices eventually give rise to banded unlink diagrams which differ by a $2$-handle-band slide; see top left of Fig.~\ref{fig:getdiagram}.) Do this for each such intersection, and then isotope $b$ vertically to $h^{-1}(3/2,2)$ (vertically isotope other bands downward in $h^{-1}(3/2,2)$ as necessary to avoid self-intersections). 

Repeat for every other index-$1$ critical point of $h\vert_\Sigma$. 
Take the bands to lie in distinct heights, by vertical isotopy. Say the bands are $b_1,\ldots, b_n$, with $3/2<h(b_1)<\cdots<h(b_n)<2$. Set $L_\Sigma:=\Sigma\cap M_{3/2}$, $v_i:=\rho_{h(b_i),3/2}$ and $v_\Sigma=\cup_i v_i$.

Note $(\K,L_\Sigma,v_\Sigma)$ is not yet a banded unlink diagram, due to the following unallowed situations that may occur when projecting the bands $b_i$ to $M_{3/2}$.

\begin{itemize}
\item It might be that an end of some $v_i$ is attached to another band $v_j$. This implies $j<i$. If so, slide the end of $v_i$ off of $v_j$ and onto either $L_\Sigma$ or $v_k$, with $k<j$. Repeat until both ends of $v_i$ are on $L_\Sigma$. There are two choices to make at each step (that is, there is a choice of which direction to slide). The two obtainable diagrams differ by a sequence of band slides. (See the left of second row of Fig.~\ref{fig:getdiagram} for the simplest case when $v_j$ has both ends on $L_\Sigma$.)
\item It might be that a band $v_i$ intersects the interior of another band $v_j$. This implies $j<i$. If so, swim $v_i$ out the length of $v_j$. If $v_j$ intersects the interior of another band $v_k$ (necessarily $k<j$), this introduces new intersections between $v_i$ and $v_k$. Repeat on each intersection of $v_i$ with another band until $v_i$ does not intersect any other bands. There are two choices at each step (that is, there is a choice of which direction to swim). The two obtainable diagrams differ by a sequence of band swims.  (See the left of third row of Fig.~\ref{fig:getdiagram} for the simplest case when $v_j$ does not intersect the interior of any other band.)
\item  It might be that a segment of a band $v_i$ passes through the descending manifold of some index-$2$-critical point of $h$. In $\K$, this means that $v_i$ intersects a $2$-handle attaching circle $C$. Swim $C$ through $v$ to remove the intersection. There are two choices of directions in which to swim. The two obtainable diagrams differ by a $2$-handle-band swim.  (See the right top row of Fig.~\ref{fig:getdiagram} for the simplest case when $v$ intersects exactly one attaching circle, exactly once.)
\item It might be that $L_\Sigma$ or a band $v_i$ still do not lie in $E(\K)$ because they intersect the ascending manifold of an index-$1$ critical point. Then push $L_\Sigma$ or $v_i$ horizontally off the ascending manifold. For each such intersection, there are two choices of which direction to push. The two obtainable diagrams differ by a dotted circle slide. (See the right second and third rows of Fig. ~\ref{fig:getdiagram}.)
\end{itemize}

Changing any choices made during this operation changes the diagram by Morse-preserving band moves. See Figure~\ref{fig:getdiagram} for a summarizing schematic.

Each move on the projections $v_i$ can be induced by a horizontal isotopy supported in a neighborhood of $b_i$. After this procedure and a vertical isotopy of the bands to $M_{3/2}$, we find an isotopy from $\Sigma$ to a surface in banded unlink position, whose banded unlink diagram we denote by $\D=(\K,L_\Sigma,v_\Sigma)$.  Via the isotopy constructed above we see that  $\Sigma$  is indeed isotopic $\Sigma(\K,L_\Sigma,v_\Sigma)$. 

We must check that the choices of $h$-regular and $h$-disjoint horizontal and vertical isotopies used to position $\Sigma$ (that is, the choice of $\Sigma'$) do not affect the resulting diagram up to Morse-preserving band moves (i.e. that $L_\Sigma$ and $v_\Sigma$ are well-defined up to Morse-preserving band moves). It is sufficient to prove the following proposition.

\begin{proposition}\label{horizontallemma}
Let $S$ and $S'$ be horizontal-vertical surfaces in $X^4$, with all minima below all bands, which are in turn below all maxima. Assume all bands of $S$ and $S'$ can be projected to $M_{3/2}$ (i.e.\ assume that the bands of $S$ and $S'$ do not intersect the ascending manifold of any index-$3$ critical point of $h$ or the descending manifold of any index-$1$ critical point of $h$).  Suppose there is an $h$-disjoint and $h$-regular horizontal or vertical isotopy $f$ taking $S$ to $S'$. Let $\D_S=(\K,L_S,v_S)$ be a banded unlink diagram obtained by setting $L_S=S\cap M_{t_0+\epsilon}$ for $t_0$ the height of the highest minima of $S$, and $v_S$ the bands obtained by projecting the bands of $S$ to $M_{3/2}$ (viewed as containing a copy of $L_S$, projected vertically) and (as in Proposition~\ref{horizontaltoband}) choosing slides, swims, $2$-handle-band slides/swims, and dotted circle slides as necessary to make the projected bands disjointly lie in $E(\K)$. Similarly choose a banded unlink diagram $\D_{S'}=(\K,L_{S'},v_{S'})$ using $S'$. Then $\D_S$ and $\D_{S'}$ are related by Morse-preserving band moves.
\end{proposition}

\begin{proof}
{\bf{Case 1: The isotopy is horizontal}}
Since the isotopy is horizontal, $L_S$ is isotopic to $L_{S'}$ in $h^{-1}(3/2)$. Therefore, $L_S$ is isotopic to $L_{S'}$ in $E(S)$ up to dotted circle slides.

Let $v_i$ be a band in $S$. The isotopy $f$ takes the banded link $\tilde{L}\cup v_1\cup\cdots\cup v_k=S\cap M_{h(v_i)}$ to the banded link $\tilde{L'}\cup v'_1\cup\cdots\cup v'_k=S'\cap M_{h(v_i)}$, up to relabeling of bands (for some isotopic links $\tilde{L}$ and $\tilde{L'}$). Since $\tilde{L}$ and $\tilde{L'}$ are isotopic, there is a natural identification between $\tilde{L}$ and $\tilde{L'}$. Say that $v_i$ goes to band $v'_i$, with their endpoints on $\tilde{L}$ identified.

Suppose $h(v_i)> 2$. As per the above argument in Proposition~\ref{horizontaltoband}, if the isotopy passes $v_i$ through the ascending manifold of an index-$2$ critical point of $h$, then this effects a $2$-handle-band slide in $\D_S$.

For any value of $h(v_i)$, if $f$ takes the ends of $\rho_{h(v_i),3/2} (v_i)$ over any other projection $\rho_{h(v_j),3/2}(v_j)$ with $h(v_j)< h(v_i)$, then as in Proposition~\ref{horizontaltoband} this effects a band slide in $\D$. If $f$ takes the interior of $\rho_{h(v_i),3/2}(v_i)$ through any other projection $\rho_{h(v_j),3/2}(v_j)$ with $h(v_j)< h(v_i)$, then as in Proposition~\ref{horizontaltoband} this effects a band swim in $\D_S$. If $f$ takes $v_i$ through the descending manifold of an index-$2$ critical point of $h$, then as in Proposition~\ref{horizontaltoband} this effects a $2$-handle-band swim in $\D_S$.

Finally, if $f$ takes $\rho_{h(v_i),3/2} (v_i)$ through the ascending manifold of an index-$1$ critical point, then as in Proposition~\ref{horizontaltoband} this effects a dotted circle slide in $\D$.

If none of the above happen to $v_i$ during $f$, then the replacement $v_i\mapsto v'_i$ just isotopes the projection of $v_i$ in  $h^{-1}(3/2)$; i.e. changes the projection of $v_i$ by isotopy in $E(K)$ and dotted circle slides.

{\bf{Case 2: The isotopy is vertical}}

By assumption, the vertical isotopy does not introduce new critical points of $h\vert_{S}$ and preserves the projections of $S$ pointwise to each $M_{t}$. Then $L_{S'}$ differs by $L_S$ by isotopy in $h^{-1}(3/2)$, i.e. isotopy in $E(K)$ and dotted circle slides. Moreover, the vertical isotopy does not affect the projections of the bands of $S$ (after identifying $L_S$ and $L_{S'}$), so these projections agree with those of $S'$. Then $\D_S$ and $\D_{S'}$ agree up to Morse-preserving band moves which arise from varying choices of how to separate the projections of bands to $M_{3/2}$ (as seen above in Proposition~\ref{horizontaltoband}). 
\end{proof}


Thus, $\D$ is well-defined from $\Sigma$ up to Morse-preserving band moves. Moreover, $\Sigma$ is isotopic to $\Sigma'$, so by Lemma~\ref{diagramfromband}, $\Sigma$ is isotopic to $\Sigma(\K,L_\Sigma,v_\Sigma)$.  This completes the proof of Proposition~\ref{horizontaltoband}.

\end{proof}

\begin{figure}
\includegraphics[width=\textwidth]{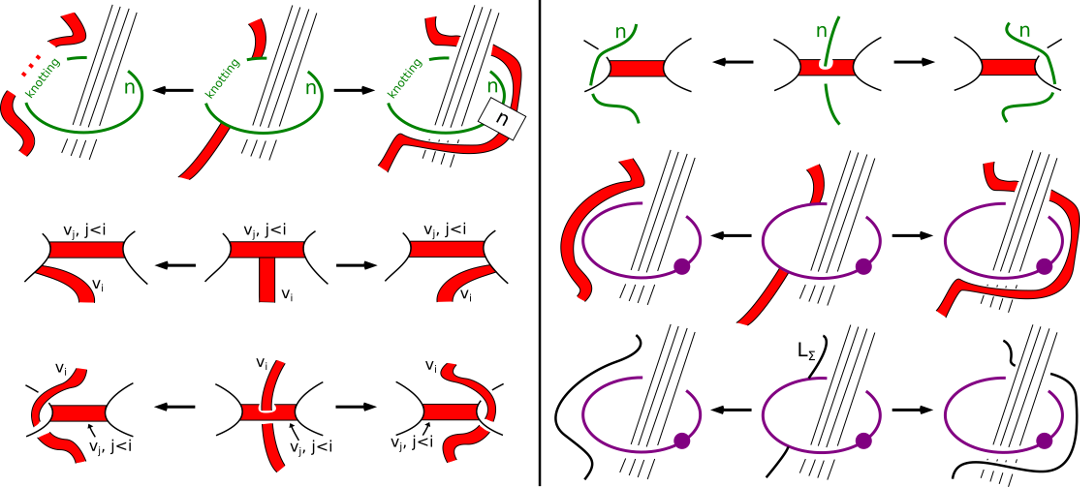}
\caption{{\bf{Left, top row:}} If a band $v_i$ intersects the ascending manifold of an index-$2$ critical point of $h$, we must choose how to slide the band off of the ascending manifold before projecting to $h^{-1}(3/2)$. In $\K$, these two choices yield diagrams that differ by a $2$-handle-band slide of $v_i$ over the corresponding $2$-handle attaching circle. {\bf{Left, second row:}} If the end of band $v_i$ lies on $v_j$, then we must choose which way to slide $v_i$ off of $v_j$. (Here we draw only the simple case that both ends of $v_j$ lie on $L_\Sigma$. We do not care about interior intersections of bands or intersections with $2$-handle attaching circles.) {\bf{Left, third row:}} If a band $v_i$ intersects the interior of band $v_j$, then we must choose which way to swim $v_i$ through and out of $v_j$. (Here we draw only the simple case that $v_j$ does not intersect any other bands. We do not care about intersections with $2$-handle attaching circles.) {\bf{Right, top row:}} If a band $v_i$ intersects the descending manifold of an index-$2$ critical point, we must choose how to swim the corresponding attaching circle in $\K$ out of the band $v_i$. (Here, we draw only the simple case that only one $2$-handle attaching circle intersects $v_i$, in one point.) {\bf{Right, second and third row:}} If a band $v_i$ or $L$ (respectively) intersect the ascending manifold of an index-$1$ critical point, then we push horizontally off. The resulting diagrams differ by a dotted circle slide.}\label{fig:getdiagram}
\end{figure}

Given $\Sigma$ as in Proposition~\ref{horizontaltoband}, we let $(\K,L_\Sigma,v_\Sigma)$ denote the banded unlink diagram resulting from the proof of Proposition~\ref{horizontaltoband}.

We now show that we need not restrict the ordering of critical points of a horizontal-vertical surface in order to obtain a banded unlink diagram.

\begin{lemma}\label{horizontalisotopy}
Let $\Sigma$ be a surface in horizontal-vertical position. Then we may obtain a banded unlink diagram $(\K,L_\Sigma,v_\Sigma)$ so that $L_\Sigma$ and $v_\Sigma$ are determined up to Morse-preserving band moves by the embedding of $\Sigma$ into $X^4$. Moreover, $\Sigma$ is isotopic to $\Sigma(\K,L_\Sigma,v_\Sigma)$.
\end{lemma}
\begin{proof}
Choose an ordering $x_1,\ldots, x_n$ of the critical points of $h\vert_\Sigma$, with all index-$0$ critical points coming before all index-$1$ critical points, which come before all index-$2$ critical points. Perform $h$-regular and $h$-disjoint horizontal and vertical isotopy to $\Sigma$ to reorder the horizontal regions according to this ordering, to obtain surface $\Sigma'$ (the vertical isotopies move horizontal regions to the appropriate height; we apply horizontal isotopy as necessary to ensure horizontal regions never intersect). Set $L_{\Sigma}=L_{\Sigma'}$ and $v_{\Sigma}=v_{\Sigma'}$.

Suppose $y_1,\ldots, y_n$ is another ordering of the critical points of $h\vert_\Sigma$ with all index-$0$ critical points coming before all index-$1$ critical points, which come before all index-$2$ critical points. Let $\Sigma''$ be the surface obtained by isotoping $\Sigma$ to reorder the horizontal regions of $\Sigma$ according to this ordering. Then $\Sigma''$ can be transformed into $\Sigma'$ by a sequence of $h$-regular and $h$-disjoint horizontal and vertical isotopies (and the surface is in horizontal-vertical position after each isotopy, with minima below bands below maxima. Essentially, we reorder minima, keeping them below all bands. Then we reorder bands, keeping them between the minima and maxima. Then we reorder the maxima, keeping them above the bands). 
By Proposition~\ref{horizontallemma}, the diagrams $(\K,L_{\Sigma'},v_{\Sigma'})$ and $(\K,L_{\Sigma''},v_{\Sigma''})$ agree up to Morse-preserving band moves. Therefore, $(\K, L_{\Sigma}, v_{\Sigma})$ does not depend on the choice of $\Sigma'$. By Proposition~\ref{horizontaltoband},  $(\K, L_{\Sigma}, v_{\Sigma})$ is well-defined up to Morse-preserving band moves. Moreover, $\Sigma(\K,L_\Sigma,v_\Sigma)$ is isotopic to $\Sigma'$, so is isotopic to $\Sigma$.
\end{proof}

Given $\Sigma$ as in Lemma~\ref{horizontalisotopy}, we let $(\K,L_\Sigma,v_\Sigma)$ denote the banded unlink diagram resulting from the proof of Lemma~\ref{horizontalisotopy}.
\begin{remark}\label{fixhorizontallemma}
Note that we can immediately extend Proposition~\ref{horizontallemma} to say that if $\Sigma$ and $\Sigma'$ are horizontal-vertical surface which are isotopic through an $h$-regular and $h$-disjoint horizontal or vertical isotopy, then $(\K,L_\Sigma,v_\Sigma)$ is related to $(\K,L_{\Sigma'},v_{\Sigma'})$ by Morse-preserving band moves.
\end{remark} 

\subsection{Interlude: surface singularities}
The following definitions generalize the classes of surface singularities identified and studied in~\cite{kearton2008all}. 

\begin{definition}[{\cite[Def. 2.5]{kearton2008all}}]\label{def:singular}
For a fixed surface $F$, let CS be the space of all smoothly embedded surfaces $\Sigma\subset X^4$ with $\Sigma\cong F$. CS inherits the induced topology from the Whitney topology on the space of smooth maps $F \rightarrow X$. From now on, when we write ``CS'', the topology of the embedded surface is implictly understood to be fixed.  
  We say that $\Sigma\in\CS$ is: 
\begin{itemize}
\item {\emph{generic}} if $h|_{\Sigma}$ is Morse and all critical points of $h|_{\Sigma}$ have distinct height values under $h$;
\item an {\emph{$A_1^+A_1^+$-singularity}} if $h|_{\Sigma}$ fails to be generic because of two nondegenerate extrema of $h$ that have the same $h$ value; 
\item an {\emph{$A_1^+A_1^-$-singularity}} if $h\vert_{\Sigma}$ fails to be generic because a nondegenerate saddle and extremum of $h$ that have the same $h$ value;
\item an {\emph{$A_1^-A_1^-$-singularity}} if $h\vert_{\Sigma}$ fails to be generic because of two nondegenerate saddles of $h$ that have the same $h$ value;
\item an {\emph{$A_2$-singularity}} if $h\vert_{\Sigma}$ fails to be generic because of a singularity of $h\vert_{\Sigma}:\Sigma\rightarrow \mathbb{R}$ having the form $h(x,y)=x^2-y^3$ in local coordinates $(x,y)$ on $\Sigma$. 
\end{itemize}

\end{definition}

Let $\S_h\subset \CS$ be the subspace of all surfaces which are $A_1^+A_1^+$-, $A_1^+A_1^-$-, $A_1^-A_1^-$-, or $A_2$-singularities. 
We call $\S_h$ the {\emph{singular subspace of $X^4$ with respect to $h$.}} 
When $\mathcal{K}=\K_0$, then $\S_h$ agrees with the singular subspace defined by Kearton-Kurlin~\cite{kearton2008all}. 

\subsection{Banded unlink diagrams for generic surfaces}\label{sec:diagramforgeneric}

%
%
%

The following lemma is analogous to~\cite[Prop. 2.6(i)]{kearton2008all}.

\begin{lemma}\label{generictoband}
 Let $\Sigma\subset X^4$ be a generic surface which is disjoint from critical points of $h$. Then we may obtain a banded unlink diagram $(\K,L_\Sigma,v_\Sigma)$ determined by $\Sigma\into X^4$ up to Morse-preserving band moves. Moreover, $\Sigma$ is isotopic to $\Sigma(\K, L_\Sigma, v_\Sigma)$.
\end{lemma}

\begin{proof}
We will isotope $\Sigma$ into horizontal-vertical position and apply Lemma~\ref{horizontalisotopy}. See Figure~\ref{fig:flatten} for a schematic of the isotopy.

Flatten $\Sigma$ in a small neighborhood of each local extrema. Say the smallest value of $t$ for which $M_t\cap\Sigma$ is nonempty is $t_0$; so $\Sigma\cap M_{t_0+\epsilon}$ is an unlink for very small $\epsilon>0$. Take the values of critical points of $h\vert_{\Sigma}$ to be $t_0< t_1<\dots< t_n$. Peturb $\Sigma$ vertically if necessary so that $\{t_0,\ldots, t_n\}\cap\{0,1,2,3,4\}=\emptyset$ (for small perturbations, the choice of perturbation does not affect the construction of Lemma~\ref{horizontalisotopy}). Since $\Sigma$ is generic, there is exactly one critical point of $h\vert_{\Sigma}$ in each $M_{t_i}$.

Fixing $\Sigma\cap M_{t_0}$, horizontally isotope $\Sigma\cap h^{-1}(t_0,4]$ so that $\Sigma$ is vertical in $h^{-1}[0,t_1-\epsilon)$.

Then the projection of $\Sigma\cap h^{-1}[t_1-\epsilon, t_1+\epsilon]$ to $M_{t_1}$ is a bounded (perhaps disconnected) surface $S$. Specifically, if $x$ is the critical point of $h\vert_{\Sigma}$ in $M_{t_1}$, then this $S$ is\[\begin{cases}\text{the disjoint union of an annulus and a disk}&\text{$x$ is a minimum}\\
\text{a planar surface with three boundary components}&\text{$x$ is a saddle}\\
\text{a disk}&\text{$x$ is a maximum.}\end{cases}.\]

If $x$ is an extremum, perform a horizontal isotopy of $\Sigma$ in $h^{-1}[t_1-\epsilon,t_1+\epsilon]$ so that $\Sigma\cap h^{-1}(t_0,t_1)$ is vertical. The cross-section $\Sigma\cap h^{-1}(t_1)$ includes one horizontal disk.

Say that $x$ is a saddle. Then up to isotopy in $M_{t_1}$, $S$ can be uniquely contracted to a $1$-complex $\Sigma\cap M_{t_1-\epsilon}\cup\{$edge $E\}$. Vertically isotope $\Sigma$ in $h^{-1}[t_1-\epsilon,t_1+\epsilon]$ so that $\Sigma$ is vertical in $h^{1}(t_0,t_1)$ and $\Sigma\cap M_{t_1}=(\Sigma\cap M_{t_1-\epsilon})\cup($band along $E$). The framing of the band along arc $E$ agrees with the surface framing $S$ induces on $E$.

Now for $i=1,\ldots, n-1$ (in order), repeat this procedure. That is, horizontally isotope $\Sigma\cap h^{-1}(t_i,4]$ so that $\Sigma$ is vertical in $h^{-1}(t_i,t_{i+1}-\epsilon)$. If the critical point of $h\vert_\Sigma$ in $M_{t_{i+1}}$ is an extrema, then horizontally isotope $\Sigma$ near $M_{t_{i+1}}$ so that $\Sigma$ is vertical in $h^{-1}(t_i,t_{i+1})$ and $\Sigma\cap M_{t_{i+1}}$ is a disjoint union of a (possibly empty) link and a disk. If the critical point of $h\vert_\Sigma$ in $M_{t_{i+1}}$ is a saddle, then again (up to isotopy in $M_{t_{i+1}}$) there is a unique band $b$ that can be attached to $\Sigma\cap M_{t_{i+1}-\epsilon}$ so that resolving $b$ yields $\Sigma\cap M_{t_{i+1}+\epsilon}$. Vertically isotope $\Sigma\cap h^{-1}[t_{i+1}-\epsilon,t_{i+1}+\epsilon]$ so that $\Sigma$ is vertical in $h^{-1}(t_i,t_{i+1})$ and $\Sigma\cap h^{-1}(t_{i+1})$ is the banded link $\Sigma\cap M_{t_{i+1}-\epsilon}\cup b$. 

Call the resulting surface $\Sigma'$, so the original $\Sigma$ is isotopic through $h$-regular and $h$-disjoint horizontal and vertical isotopies to $\Sigma'$, where $\Sigma'$ is in horizontal-vertical position. Suppose $\Sigma''$ is another horizontal-vertical surface isotopic to $\Sigma$ through $h$-disjoint and $h$-regular horizontal and vertical isotopies. Then $\Sigma'$ and $\Sigma''$ are isotopic through $h$-disjoint and $h$-regular horizontal and vertical isotopies. By Proposition~\ref{horizontallemma} and Remark~\ref{fixhorizontallemma}, $(\K,L_{\Sigma'},v_{\Sigma'})$ is related to $(\K,L_{\Sigma''},v_{\Sigma''})$ by Morse-preserving band moves. 


Set $(\K,L_{\Sigma},v_{\Sigma}):=(\K,L_{\Sigma'},v_{\Sigma'})$. By Lemma~\ref{horizontalisotopy}, $\Sigma$ determines a banded unlink diagram $(\K,L_\Sigma,v_\Sigma)$ well-defined up to Morse-preserving band moves. Moreover, $\Sigma$ is isotopic to $\Sigma(\K,L_\Sigma,v_\Sigma)$.
\end{proof}

\begin{figure}
\includegraphics[width=\textwidth]{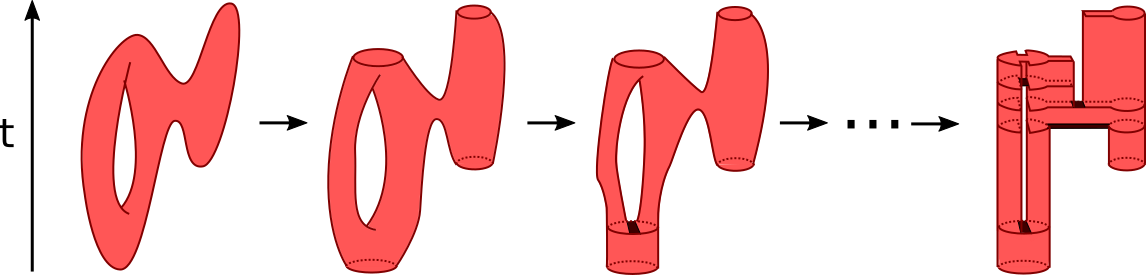}
\caption{To isotope a generic surface $\Sigma\subset X^4$ into horizontal-vertical position with respect to $h$, we first flatten a neighborhood of each extrema of $h\vert_\Sigma$. Then we isotope $\Sigma$ to be vertical below each critical point of $h\vert_\Sigma$, starting from the lowest critical point and working upward. We give more detail in the proof of Lemma~\ref{generictoband}.}\label{fig:flatten}
\end{figure}

Given $\Sigma$ as in Lemma~\ref{generictoband}, we let $(\K,L_\Sigma,v_\Sigma)$ denote the banded unlink diagram resulting from the proof of Lemma~\ref{generictoband}.  We say that $(\K,L_\Sigma,v_\Sigma)$ {\emph{describes}} or {\emph{is a banded unlink diagram for}} $\Sigma$.
\begin{remark}
Note that when $\Sigma$ is a generic surface in banded unlink position, the definitions of $(\K,L_\Sigma, v_\Sigma)$ in Lemmas~\ref{diagramfromband} and~\ref{generictoband} agree.
\end{remark}

\subsection{Banded unlink diagrams for arbitrary surfaces}\label{sec:arbitrary}

The following Lemma is analogous to~\cite[Prop. 2.6(ii-iv)]{kearton2008all}.

\begin{lemma}\label{lemma:simpleiso}
Let $\Sigma$ and $\Sigma'$ be generic surfaces in $X^4$. Let $\D=(\K,L,v)$ and $\D'=(\K,L',v')$ be banded unlink diagrams associated to $\Sigma$ and $\Sigma'$ respectively, as in Lemma~\ref{generictoband}.
\begin{enumerate}
\item If $\Sigma$ and $\Sigma'$ are $h$-disjoint isotopic through generic surfaces, then $\D$ and $\D'$ are related by Morse-preserving band moves.
\item If $\Sigma$ and $\Sigma'$ are $h$-disjoint isotopic through generic surfaces and one $A_1^{\pm} A_1^{\pm}$-singularity, then $\D$ and $\D'$ are related by Morse-preserving band moves.
\item If $\Sigma$ and $\Sigma'$ are $h$-disjoint isotopic through generic surfaces and one $A_2$-singularity, then $\D$ and $\D'$ are related by band moves.
\end{enumerate}
\end{lemma}

\begin{proof}
\noindent (i) An isotopy through generic surfaces preserves the level sets $\Sigma\cap h^{-1}(t)$ up to isotopy and reparametrization of $h$. Therefore, this isotopy does not affect the construction of Lemma~\ref{generictoband}. The unlink $L$ is taken to $L'$ by isotopy in $h^{-1}(3/2)$, so $L'$ can be obtained from $L$ by isotopy in $E(\K)$ and dotted circle slides. Each band projection is then isotoped in $E(K)$ unless the corresponding band meets an:
\begin{itemize}
\item ascending manifold of an index-$2$ critical point of $h$, inducing a $2$-handle-band slide,
\item descending manifold of an index-$2$ critical point of $h$, inducing a $2$-handle-band swim,
\item ascending manifold of an index-$1$ critical point of $h$, inducing a dotted circle slide.
\end{itemize}


\noindent (ii) Perturb the isotopy so that near the singularity, the isotopy is a vertical exchange of heights between two critical points. Say the isotopy goes from $\Sigma$ to $\Sigma_0$, then vertically to $\Sigma_1$, and then to $\Sigma'$, where the $A_1^{\pm} A_1^{\pm}$ singularity appears during the isotopy from $\Sigma_0$ to $\Sigma_1$. We consider each of the following possibilities: 
\begin{itemize}
\item If the exchanged critical points are both extrema, then this does not affect the construction of Lemma~\ref{generictoband}.
\item If the exchanged critical points are an extremum and an index-$1$ critical point, then this does not affect the construction of Lemma~\ref{generictoband}. (These critical points correspond to a minimum or maximum disk and a band which does not intersect the interior of that disk.)
\item If the exchanged critical points are both index-$1$, then they correspond to bands whose projections to $M_{3/2}$ must be disjoint (since at some point during this vertical isotopy, they live in a common $M_t$). Therefore, this does not affect the diagram resulting from Lemma~\ref{generictoband}.
\end{itemize} 
Therefore, $\Sigma_0$ and $\Sigma_1$ have banded unlink diagrams equivalent up to Morse-preserving band moves. The claim follows from part (i).

\noindent (iii) When passing an $A_2$-singularity (away from critical points of $h$) a nondegenerate saddle and extremum appear or disappear~\cite[Claim 4.3(iv)]{kearton2008all}.  In the case when the extremum created is a minimum this corresponds to performing a cup move, while in the case of a maximum the banded unlinks are related by a cap move.  Away from the $A_2$-singularity, this is an isotopy through generic surfaces, so the claim follows from part (i).
\end{proof}


The following analysis appears in~\cite{kearton2008all}. Although they state this lemma in $S^4$ rather than an arbitrary $4$-manifold $X^4$, their techniques hold generally.

\begin{lemma}[{\cite[Claim 4.3]{kearton2008all}}]
\label{lemma:generic}\
\begin{enumerate}
\item The subspace $\overline{\S_h}$ is codimension-$1$ in $\CS$.
\item Every element of $\CS\setminus\overline{\S_h}$ is generic.
\item Any $h$-disjoint isotopy of a surface in $X^4$ can be deformed to an $h$-disjoint isotopy so that all intermediate surfaces are generic except for finitely many singularities as in Definition~\ref{def:singular}.
\end{enumerate}
\end{lemma}

\begin{lemma}\label{lemma:regular}
Let $\Sigma$ and $\Sigma'$ be surfaces which are disjoint from the critical points of $h$. Let $f$ be an isotopy with $f( F\times I)=\Sigma$ and $f( F\times 1)=\Sigma'$. Then $f$ can be deformed to an $h$-disjoint isotopy, fixing $f\vert_{ F\times 0}$ and $f\vert_{ F\times 1}$.
\end{lemma}
\begin{proof}
Let $H\subset X^4$ be the critical points of $h$. Then $ F\times I$ is a smooth codimension-$2$ submanifold of $X^4\times I$, while $H\times I$ is a dimension-$1$ submanifold of $X^4\times I$. We may generically perturb $ F\times I$ (and hence $f$) rel boundary to be disjoint from $H\times I$, to obtain an $h$-disjoint isotopy.
\end{proof}

\begin{lemma}\label{finalmainlemma}
Let $\Sigma$ and $\Sigma'$ be isotopic generic surfaces embedded in $X^4$ which are both disjoint from critical points of $h$. Say that $\Sigma$ and $\Sigma'$ have banded unlink diagrams $\D$ and $\D'$ respectively. Then $\D$ can be transformed into $\D'$ by band moves.
\end{lemma}
\begin{proof}
 Let $f: F\times I\to X^4\times I$ be an isotopy from $\Sigma=f( F\times 0)$ to $\Sigma'=f( F\times 1)$ By Lemma~\ref{lemma:regular}, $f$ can be taken to be $h$-disjoint. By Lemma~\ref{lemma:generic}, $f$ can be perturbed slightly so that $f( F\times s)$ is generic except for finitely many values of $s$ (at which time $f( F\times s)$ is a singularity as in Def.~\ref{def:singular}),  and $f$ is still $h$-disjoint. Fix  $0<s_1<\cdots<s_n<1$ so that $f( F\times s)$ is generic if $t\not\in\{s_i\}$, and $f( F\times s_i)$ is a singularity as in Definition~\ref{def:singular}.

For $i=1,\ldots, n-1$, let $\Sigma_i:=f( F\times(s_i+\epsilon))$. Let $\Sigma_0:=\Sigma$ and $\Sigma_n:=\Sigma'$. Let $\D_i=(\K,L_{\Sigma_i},v_{\Sigma_i})$. By Lemma~\ref{lemma:simpleiso}, $\D_i$ is obtained from $\D_{i-1}$ by band moves. Thus, $\D_n=\D'$ is obtained from $\D_0=\D$ by band moves.
\end{proof}

Any surface in $X^4$ can be perturbed to be generic and away from critical points of $h$. Therefore, Lemma~\ref{finalmainlemma} allows us to make the following definition, completing the proof of Theorem~\ref{thm:generalswenton}.

\begin{definition}
Let $\Sigma$ be a surface embedded in $X^4$. Let $\Sigma'$ be a generic surface which is disjoint from critical points of $h$ so that $\Sigma$ is isotopic to $\Sigma'$. We say that $(\K,L_{\Sigma'},v_{\Sigma'})$ {\emph{describes}} or {\emph{is a banded unlink diagram for}} $\Sigma$. By Lemma~\ref{finalmainlemma}, this diagram is well-defined up to band moves.
\end{definition}


\label{endsec3}

\section{Uniqueness of bridge trisections}\label{sec:bridge}

First, we recall the definition of trisection of a closed $4$-manifold.
\begin{definition}[\cite{gay2016trisecting}]
Let $X^4$ be a closed $4$-manifold. A {\emph{$(g,k)$-trisection}} of $X^4$ is a triple $(X_1,X_2,X_3)$ where
\begin{itemize}
\item $X_1\cup X_2\cup X_3=X^4$,
\item $X_i\cong\natural_{k_i} S^1\times B^3$,
\item $X_i\cap X_j=\boundary X_i\cap\boundary X_j\cong \natural_g S^1\times B^2$
\item $X_1\cap X_2\cap X_3\cong \Sigma_g$,
\end{itemize}

where $\Sigma_g$ is the closed orientable surface of genus $g$. Here, $g$ is an integer while $k=(k_1,k_2,k_3)$ is a triple of integers. If $k_1=k_2=k_3$, then the trisection is said to be {\emph{balanced}}.
\end{definition}

Briefly, a trisection is a decomposition of a $4$-manifold into three elementary pieces, analogous to a Heegaard splitting of a $3$-manifold into two elementary pieces. Intuitively, one should think that the need for an ``extra'' piece of this decomposition when the dimension increases corresponds to an ``extra'' type of handle. That is, given a Heegaard splitting $M^3=H_1\cup H_2$, one can view $H_1$ as containing the $0$- and $1$-handles of $M^3$ while $H_2$ contains the $2$- and $3$-handles of $M^3$. Similar is true for a trisection $(X_1,X_2,X_3)$ of $X^4$; one can view $X_1$ as containing the $0$- and $1$-handles of $X^4$ and $X_3$ as containing the $3$- and $4$-handles of $X^4$ while $X_2$ contains the $2$-handles of $X^4$. See~\cite{meier2016propertyr} for a clear description of a trisection from this point of view.

Note that from the definition, $(\Sigma_g, X_i\cap X_j, X_i\cap X_k)$ gives a Heegaard splitting of $\partial X_i$. By Laudenbach and Po\'{e}naru~\cite{laudenbach1972handlebodies}, $X^4$ is specified by its {\emph{spine}}, $\Sigma_g\cup_{i,j}(X_i\cap X_j)$. Therefore, we usually describe a trisection $(X_1,X_2,X_3)$ by a {\emph{trisection diagram}} $(\Sigma_g,\alpha,\beta,\gamma)$ where each of $\alpha,\beta,$ and $\gamma$ consist of $g$ independent curves bounding disks in the handlebodies $X_1\cap X_2, X_2\cap X_3, X_1\cap X_3$ respectively.

We do not require much knowledge about trisections for this paper. For more exposition of trisections, refer to~\cite{gay2016trisecting}.

\begin{definition}
The {\emph{standard trisection of $S^4$}} is the unique $(0,0)$-trisection $(X_1,$ $X_2,$ $X_3)$. View $S^4=\R^4\cup\infty$, with coordinates $(x,y,r,\theta)$ on $\R^4$, where $(x,y)$ are Cartesian planar coordinates of a plane and $(r,\theta)$ are polar planar coordinates. Up to isotopy, $X_i=\{\theta\in[2\pi/3\cdot i, 2\pi/3\cdot(i+1)]\}\cup\infty$. Then $X_i\cong B^4$, $X_i\cap X_{i+1}=\{\theta=2\pi/3\cdot(i+1)\}\cup\infty\cong B^3$, and $X_1\cap X_j\cap X_k=\{r=0\}\cup\infty\cong S^2$.
\end{definition}

In~\cite{meier2017bridge}, Meier and Zupan introduce bridge trisections of surfaces in $S^4$. In~\cite{meier2018bridge}, they extend this notion to surfaces in an arbitrary closed $4$-manifold.

\begin{definition}[\cite{meier2017bridge,meier2018bridge}]
Let $S$ be a surface embedded in $X^4$. Let $\tri=(X_1,X_2,X_3)$ be a trisection of a closed $4$-manifold $X^4$. We say that $S$ is in $(c,b)$-bridge position with respect to $\tri$ if
\begin{itemize}
\item $S\cap X_i$ is a disjoint union of $c$ boundary parallel disks,
\item $S\cap X_i\cap X_j$ is a trivial tangle of $b$ arcs.
\end{itemize}

Here, $b$ is an integer and $c=(c_i,c_j,c_k)$ is a triple of integers. Note $\chi(S)=\sum c_i-b$. 
\end{definition}

See Figure~\ref{fig:rp2trisection} for an example of a surface in bridge position.

\begin{figure}
\includegraphics[width=.75\textwidth]{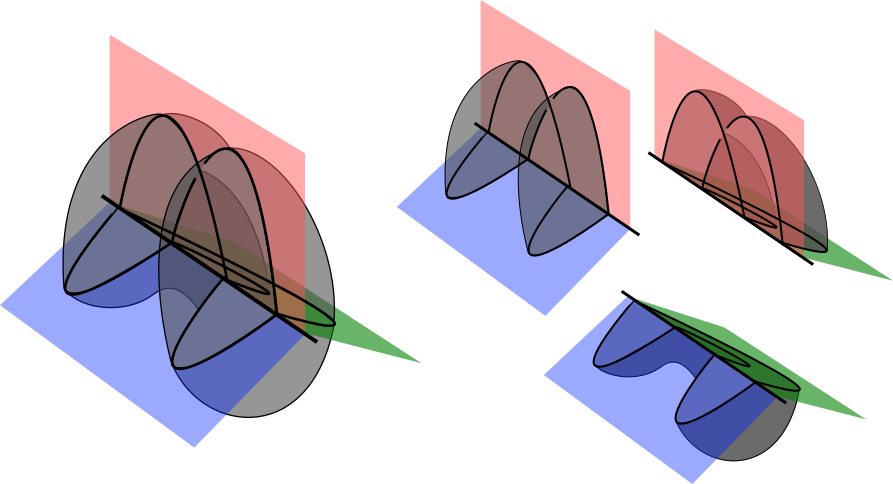}
\caption{An $\RP^2$ in $S^4$ in $(1,1,1;2)$-bridge position with respect to the standard trisection. Left: We draw the whole surface (projected to $\R^3$. The three planes indicate the three $3$-balls $\{X_i\cap X_j\}$. Right: We show each disk system (each have one component) individually.}\label{fig:rp2trisection}
\end{figure}

\begin{theorem}[\cite{meier2017bridge,meier2018bridge}]\label{thm:triunique}
Let $S$ be a surface embedded in $X^4$ with trisection $(X_1,X_2,X_3)$. Then for some $c$ and $b$, $S$ can be isotoped into $(c,b)$-bridge position with respect to $\tri$. We may take $c_1=c_2=c_3$.
\end{theorem}

Because a collection of boundary parallel disks in $\natural (S^1\times B^3)$ is uniquely determined by its boundary (up to isotopy rel boundary), a surface $S$ in bridge position is determined up to isotopy by $S\cap(\cup_{i\neq j}X_i\cap X_j)$. 

\begin{definition}
Let $X^4$ be a $4$-manifold with trisection $\tri(X_1,X_2,X_3)$.
We say that an isotopy $f$ of $X^4$ is {\emph{$\tri$-regular}} if $f_t(X_i)=X_i$ for each $i=1,2,3$ for all $t$.
\end{definition}

There is a natural perturbation of a surface in bridge position, analogous to perturbation of a knot in bridge position within a $3$-manifold.

\begin{definition}[\cite{meier2017bridge,meier2018bridge}]
Let $S\subset X^4$ be a surface in $(c,b)$-bridge position with respect to $\tri=(X_1,X_2,X_3)$. Let $\Delta\subset X_i\setminus\nu(S)$ be a properly embedded disk so that $\boundary\Delta$ consists of one arc in $\boundary\nu(S)$, one arc in $X_i\cap X_{i+1}$, and one arc in $X_i\cap X_{i-1}$. Obtain $S'$ by compressing $S$ along $\Delta$. Note $S'$ is isotopic to $S$ and is in $(c',b+1)$ bridge position, where $c'_i=c_i+1, c'_{i+1}=c_{i+1},c'_{i-1}=c_{i-1}$. We say that $S'$ is obtained from $S$ by elementary perturbation, while $S$ is obtained from $S'$ by elementary deperturbation (see Figure~\ref{fig:perturb}).
\end{definition}

\begin{figure}
\includegraphics[width=.9\textwidth]{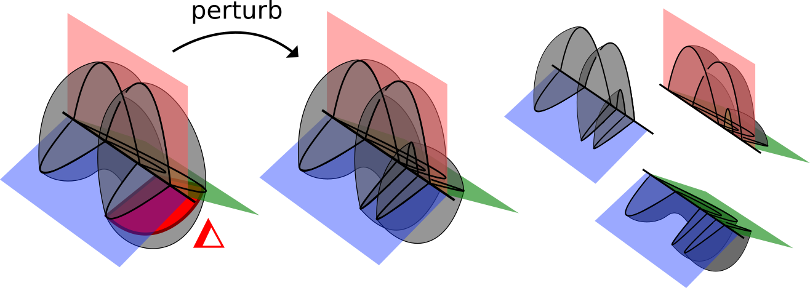}
\caption{Left: $\Sigma\cong\RP^2$ in $(1,1,1;2)$-bridge position in $S^4$ (with respect to the standard trisection). We indicate a disk $\Delta$ along which we perturb. Right: after perturbing, $\Sigma$ is in $(1,2,1;3)$-bridge position (up to permuting $X_1,X_2,X_3$).}\label{fig:perturb}
\end{figure}

Theorem~\ref{thm:triunique} shows existence of bridge trisections. The following theorem of~\cite{meier2017bridge} gives uniqueness of bridge trisections with respect to the standard trisection of $S^4$.

\begin{theorem}[\cite{meier2017bridge}]\label{s4theorem}
Let $S$ and $S'$ be surfaces in bridge position with respect to the standard trisection $\tri_0$ of $S^4$. Suppose $S$ is isotopic to $S'$. Then $S$ can be taken to $S'$ by a sequence of perturbations and deperturbations, followed by a $\tri_0$-regular isotopy.
\end{theorem}

Theorem~\ref{s4theorem} relies on Theorem~\ref{thm:swenton}, which is specific to $S^4$. In~\cite{meier2017bridge}, Meier and Zupan give an equivalence between bridge trisections and banded unlink diagrams, and then show how to translate moves on banded unlink diagrams into sequences of perturbations and deperturbations. These moves do not occur in order; in particular they do not show that all the deperturbations can come after the perturbations.

In~\cite{meier2018bridge}, Meier and Zupan state the following theorem as a conjecture, and comment that they believe it would follow from a generalized version of Theorem~\ref{thm:swenton} similarly to the proof of Theorem~\ref{s4theorem}. We will prove this theorem using Theorem~\ref{thm:generalswenton}.

\begin{theorem}\label{bridgethm}
Let $S$ and $S'$ be surfaces in bridge position with respect to a trisection $\tri$ of a closed $4$-manifold $X^4$. Suppose $S$ is isotopic to $S'$. Then $S$ can be taken to $S'$ by a sequence of perturbations and deperturbations, followed by a $\tri$-regular isotopy.
\end{theorem}

Before proving Theorem~\ref{bridgethm}, we state several necessary definitions and lemmas from~\cite{meier2018bridge}.
%

\begin{definition}
Let $T$ be a tangle of properly embedded arcs in a solid handlebody $H$. We say that $T$ is {\emph{trivial}} if $T$ is boundary parallel, i.e. cobounding disjoint disks $D$ with arcs $T'\subset\boundary H$. We call $T'$ a {\emph{shadow}} of $T$. 
\end{definition}

\begin{definition}[\cite{meier2017bridge, meier2018bridge}]
Let $T$ be a trivial tangle in a handlebody $H$ with shadow $T'$. Let $v$ be a set of bands attached to $T$, with core arcs $\eta$ disjoint from a core of $H$. Project $\eta$ to $\boundary H$. We call $\eta$ a {\emph{shadow}} of $v$. We say that $\eta$ is \emph{dual} to $T'$ if $\mathring{\eta}\cap T'=\emptyset$ and each component of $\eta\cup T'$ is simply connected.
\end{definition}

\begin{notation}
Given a Kirby diagram $\K$, let $L_1\subset S^3\supset\K$ be the unlink of dotted circles (defining $1$-handles) in $\K$. Recall that $M_{3/2}= S^3$ surgered along $L_1$ with $0$-framing. Let $L_2\subset S^3$ be the link of $2$-handle attaching circles in $\K$.

Recall $E(\K)=S^3\setminus\nu(L_1\cup L_2)$.
\end{notation}

\begin{definition}[\cite{meier2018bridge}]\label{def:bridge}
Let $\mathcal{K}$ be a Kirby diagram. Let $H\cup_{F} H'$ be a Heegaard splitting of $M_{3/2}$ so that a core of $H$ contains $L_2$ and a core of $H'$ contains $L_1$.

Let $(\K,L,v)$ be a banded unlink. We say that $(\K,L,v)$ is in {\emph{bridge position}} with respect to the Heegaard splitting $H\cup_F H'$ if the following are true:
\begin{itemize}
\item $L\cap H$ and $L\cap H'$ are each trivial tangles with no closed components,
\item The bands $v$ are all contained in $H$ and there is a shadow $\eta$ of $v\subset H$ so that the surface framing $\boundary H$ induces on $\eta$ agrees with the framing $v$ induces on $\eta$,
\item There is a shadow $L'$ of $L\cap H$ so that $\eta$ and $L'$ are dual.
\end{itemize}
\end{definition}

See Figure~\ref{fig:bridgeposition} for an example of a banded unlink in bridge position. Meier and Zupan~\cite{meier2018bridge} show that every banded unlink can be put into bridge position with respect to a given Heegaard splitting of $M_{3/2}$ (the proof is similar to the fact that every knot in a Heegaard-split $3$-manifold can be isotoped into bridge position).

\begin{figure}
\includegraphics[width=.75\textwidth]{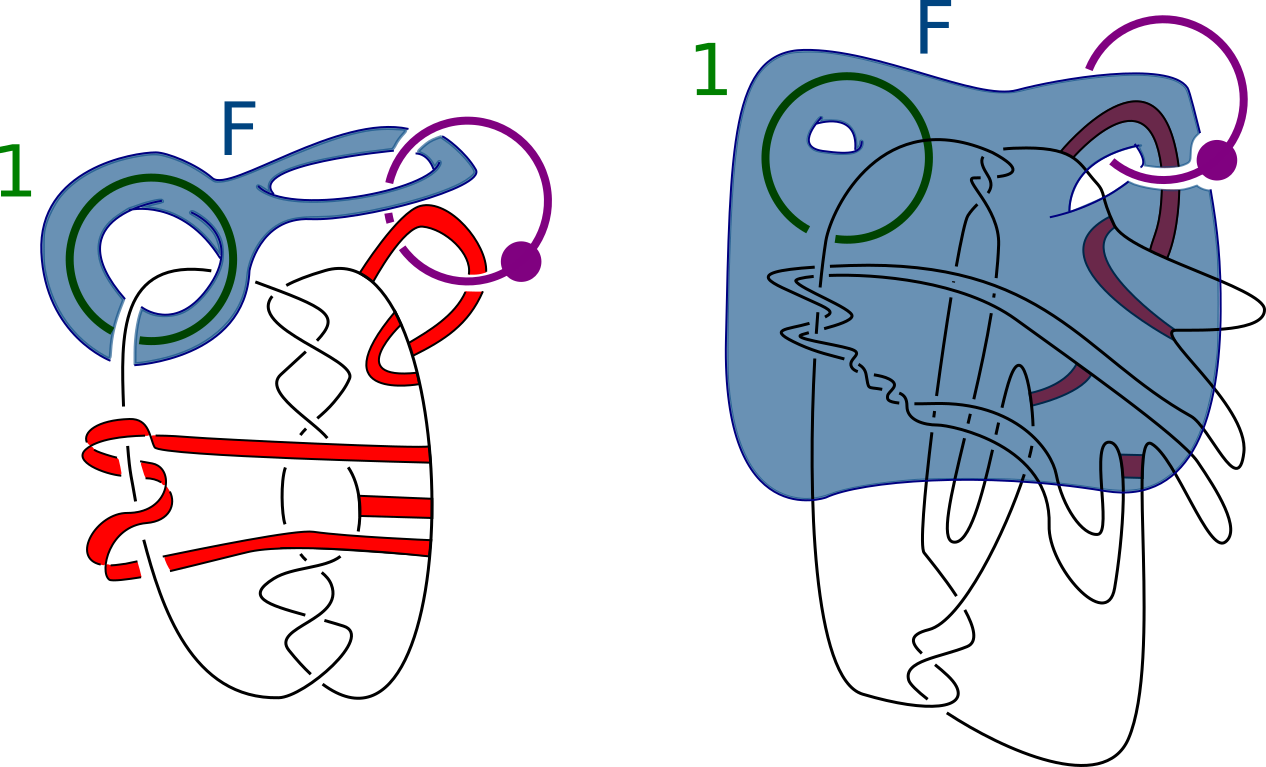}
\caption{Left: A banded unlink for a torus in $\CP^2\#(S^1\times S^3)$. The surface $F$ induces a genus-$2$ Heegaard splitting $H\cup_F H'$ on $M_1=S^1\times S^2$. Right: We isotope the banded unlink into bridge position with respect to the Heegaard splitting $H\cup_F H'$, as in Def.~\ref{def:bridge}.}\label{fig:bridgeposition}
\end{figure}

\begin{lemma}{\cite[Lemma 4.4]{meier2018bridge}}\label{lemma:bridge}
Let $(\K,L,v)$ be a banded unlink diagram. Let $H\cup_F H'$ be a Heegaard splitting of $M_1$. After isotopy of $L\cup v$ in $E(\K)$, we may assume $(\K,L,v)$ is in bridge position with respect to $H\cup_F H'$.
\end{lemma}

\begin{lemma}{\cite[Lemma 4.5]{meier2018bridge}}\label{lemma44}
Let $\tri=(X_1,X_2,X_3)$ be a trisection of a $4$-manifold $X^4$. Fix a self-indexing Morse function $h:X^4\to I$, so that $X_1$ contains all the index-$0$ and -$1$ critical points of $h$, $X_2$ contains all the index-$2$ critical points, and $X_3$ contains all the index-$3$ and -$4$ critical points (see~\cite[Lemma 14]{gay2016trisecting}).

Let $\mathcal{K}$ be the Kirby diagram of $X^4$ induced by $h$, so $M_1\cong\boundary X_1$ comes with a Heegaard splitting $M_{3/2}=H\cup_F H'$, where $H=X_1\cap X_2$, $H'=X_1\cap X_3$, and $F=X_1\cap X_2\cap X_3$. (We will say that $\K$ is a Kirby diagram of $X^4$ induced by $\tri$.)

Let $(\K,L,v)$ be a banded unlink diagram describing a surface $S\subset X^4$. If $(\K,L,v)$ is in bridge position with respect to $H\cup_F H'$, then the Heegaard splitting $H\cup_F H'$ induces a trisection $\tri$ on $X^4$ so that $S$ is naturally in bridge position with respect to $\tri$. 
Similarly, if $S'$ is a surface in bridge position with respect $\tri$, then we may obtain a banded unlink $(\K,L',v')$ for $S'$ which is in bridge position with respect to the Heegaard splitting $M_{3/2}=H\cup_F H'$.
\end{lemma}

Now we are ready to prove Theorem~\ref{bridgethm}, mirroring the proof of Theorem~\ref{s4theorem} in~\cite{meier2017bridge}.

\begin{proof}
Let $S$ and $S'$ be isotopic surfaces in $X^4$ which are both in bridge position with respect to a trisection $\tri=(X_1,X_2,X_3)$. Let $\K$ be a Kirby diagram for $X^4$ as in Lemma~\ref{lemma44}. Again let $H=X_1\cap X_2$ and $H'=X_1\cap X_3$ so that $H\cup_F H'$ is a Heegaard splitting of $M_{3/2}$.


By Lemma~\ref{lemma44}, $\tri$ induces a banded unlink diagram $\D:=(\K,L,v)$ for $S$, where $\D$ is in bridge position with respect to the Heegaard splitting $H\cup_F H'$. Similarly, $\tri$ induces a banded unlink diagram $\D'=(\K,L',v')$ for $S'$ which is in bridge position with respect to $H\cup_F H'$.

By Theorem~\ref{thm:generalswenton}, $\D$ can be taken to be isotopic to $\D'$ after performing cap/cup, band slide, band swim, $2$-handle-band slide, dotted circle slide, and $2$-handle-band swim moves. Meier and Zupan show explicitly how to achieve the cap/cup, band slide, and band swim moves by perturbing and deperturbing $S$ with respect to $\tri$ in~\cite[Theorem 1.6]{meier2017bridge}.

Suppose $\D'$ is obtained from $\D$ by a single $2$-handle-band slide or swim. Let $z$ be the framed arc from a band $\nu\in v_1$ to a $2$-handle attaching circle $C\subset L_2\subset M_{3/2}$ along which the slide or swim takes place. Take $z$ disjoint from the cores of $H$ and $H'$, then isotope $z$ vertically into $H$, possibly stabilizing the bridge position of $(\K,L,v)$. This induces perturbation of $S$.
%

Thus, performing the $2$-handle-band slide or swim on $\D$ induces isotopy on $S$ which fixes $S\cap(X_1\cap X_2)$ and $S\cap (X_1\cap X_3)$ and isotopes $S\cap (X_2\cap X_3)$ within $X_2\cap X_3$. This can therefore be taken as a $\tri$-regular isotopy.
 (Diagramatically, the $2$-handle-band slide or swim induces disk-slides on a bridge trisection diagram of $S$. We do not consider the diagrammatic point of view on bridge trisections in this paper; see~\cite{meier2018bridge}.)
 
\begin{claim}
Aafter a sequence of perturbations and deperturbations of $S$, we may take $(L,v)$ to be isotopic to $(L',v')$ in $E(\K)$ up to dotted circle slides.
\end{claim}
\begin{proof}
The claim almost follows from Theorem \ref{thm:generalswenton}, except that we did not show we could take the dotted circle slides to happen after the other band moves (except for isotopy in $E(\K)$, which we did implicitly show could be taken to happen at the end of the equivalence from $(\K,L,v)$ to $(\K,L',v')$). Two diagrams that differ by dotted slides agree up to isotopy in $S(\K):=(S^3\setminus\nu(L_2))_0(L_1)$; i.e. the $3$-manifold with boundary obtained from $\K$ by deleting the $2$-handle attaching circles and surgering the dotted circles.  Therefore, if there is a sequence of (non-isotopy) band moves $m_1m_2\cdots m_n$ which takes $(\K,L,v)$ to $(\K,L',v')$ up to isotopy in $E(\K)$, then we may delete all of the dotted circle slides to find a sequence of (non-isotopy and non-dotted circle slide) band moves $m'_1m'_2\cdots m'_{n'}$ from $(\K,L,v)$ to $(\K,L',v')$ up to isotopy in $S(\K)$. Then there exist dotted circle slides $s_1,\ldots, s_k$ so that $m'_1m'_2\cdots m'_{n'}s_1\cdots s_k$ takes $(\K,L,v)$ to $(\K,L',v')$ up to isotopy in $E(\K)$.
\end{proof}

Thus, after the above perturbations and deperturbations of $S$, we may take $(L,v)$ to be isotopic to $(L',v')$ in $S(\K)$. 
 The isotopy may not respect the bridge splitting with respect to $H\cup_F H'$. By the same argument as~\cite[Theorem 1.6]{meier2017bridge} (this cites~\cite{zupan2013bridge}, which is stated for bridge splittings of a tangle in a punctured $3$-sphere but works just as well for a punctured handlebody), we may perturb and deperturb $\D_1$ (and isotope a neighborhood of $F$, taking $F$ to $F$ setwise) so that $L\cap F= L'\cap F$, $(L\cup v)\cap H$ is isotopic to $(L'\cup v')\cap H)$ rel boundary in $H\cap E(K)$, and $L\cap H'$ is isotopic to $L\cap H'$ rel boundary in $H'$. Thus, after the listed sequence of perturbations and deperturbations, we find that $S$ is $\tri$-regular isotopic to $S'$.
\end{proof}
\label{endsec4}

\section{Examples in $\CP^2$}\label{sec:cp2}

In this section, we construct isotopies of surfaces embedded in $\CP^2$. In particular, we study unit surfaces.

\begin{definition}
Let $\Sigma$ be a surface in $\CP^2$. If $\Sigma$ intersects the standard $\CP^1\subset \CP^2$ in exactly one point, then we say that $\Sigma$ is a {\emph{unit surface}}.
\end{definition}

One motivation for studying unit surfaces is to understand the Gluck twist operation~\cite{gluck1961embedding}. This is a surgery operation on a $2$-sphere $\Sigma \subset X^4$ as long as $\Sigma$ has trivial normal bundle. In particular, the Gluck twist on $S^4$ about any embedded $2$-sphere yields a homotopy $4$-sphere. The homotopy $4$-sphere resulting from a Gluck twist along $\Sigma\subset S^4$ is known to be diffeomorphic to $S^4$ for many families of $\Sigma$, including ribbon knots~\cite{gluck1961embedding, yanagawa1969ribbon}, spun knots~\cite{gluck1961embedding}, twist-spun knots~\cite{gordon1976foursphere}, band-sums of ribbon and twist-spun knots~\cite{habiro2000gluck}, and knots $0$-concordant to any of the others in this list~\cite{melvin1977thesis}. We will define each of these families later in this section.

If $\Sigma$ is a sphere in $S^4$, then we can take the connected sum of the pairs $(S^4,\Sigma)$ and $(\CP^2,\CP^1)$ to obtain the 4-manifold $\CP^2 = S^4 \#\CP^2 $ and an embedded surface which we denote by $\Sigma\# \CP^1$.  Melvin~\cite{melvin1977thesis} showed that the Gluck twist along $\Sigma \subset S^4$ is diffeomorphic to $S^4$ if and only if there is a pairwise diffeomorphism from $(\CP^2, \Sigma \#\CP^1)$ to $(\CP^2, \CP^1)$. This in part motivates the following questions of Melvin and Gabai.

\begin{question}\label{gluckquestion}
Let $F\subset\CP^2$ be a sphere in the generating homology class $[\CP^1]\in H_2(\CP^2;\Z)$ with $F\cap\CP^1=\left\{\pt\right\}$.
\begin{align*}
\text{(i)}&\text{ Is $(\CP^2,F)$ diffeomorphic as a pair to $(\CP^2,\CP^1)$?~\cite{melvin1977thesis} }\\
\text{(ii)}&\text{ Is $F$ isotopic to the standard $\CP^1$ in $\CP^2$?~\cite[Question 10.17.i]{gabai2017lightbulb} }
\end{align*}

Note that by the previously stated work of~\cite{melvin1977thesis}, Question \ref{gluckquestion}(i) is equivalent to Kirby problem 4.23~\cite{kirbyproblems} (``Is the Gluck twist of $S^4$ about an arbitrary $2$-sphere diffeomorphic to $S^4$?'').
\end{question}

In this section, we will show that many of these unit surfaces (including all the examples listed above) are in fact isotopic to the standard $\CP^1$ using the moves of Theorem~\ref{thm:generalswenton}.

First, we give an alternate definition of $\Sigma\#\CP^1$.

\begin{definition}\label{def:unitsurface}
Let $\Sigma$ be a surface in $S^4$. Let $x$ be a point in $S^4$ far from $\Sigma$, so that $\Sigma\subset S^4\setminus\nu(x)\cong B^4$. We can view $\Sigma$ as living in $\CP^2$, inside the $4$-ball $\CP^2\setminus\nu(\CP^1)$. Let $h$ be the radial Morse function on $B^4$ and isotope $\Sigma$ so that $h\vert_\Sigma$ has a unique global maximum at $y\in\Sigma$. Let $\gamma$ be an arc from $y$ extending radially outward in $B^4$ until reaching $\CP^2$.

Let $U_\Sigma$ be a copy of $\Sigma$ tubed to $\CP^1$ along an arc $\gamma$ so that $\gamma\cap(\CP^2\setminus\nu(\CP^1))$ and $\gamma\cap\nu(\CP^1)$ are each single intervals. We call $U_\Sigma$ the {\emph{unit surface associated to $\Sigma$.}} We write $U_\Sigma=\Sigma\#\CP^1$.

Equivalently, $(S^4,\Sigma)\#(\CP^2,\CP^1)=(\CP^2,U_\Sigma)$.
\end{definition}
\begin{remark}
In Definition~\ref{def:unitsurface}, the identification of $S^4\setminus\nu(x)$ with $\CP^2\setminus\nu(\CP^1)$ does not affect the embedding $\Sigma\subset\CP^2$ up to ambient isotopy. The framing of $\gamma$ does not matter so long as the orientation on $U_\Sigma$ agrees with the orientations on $\Sigma$ and $\CP^1$.
 \end{remark}
 
\begin{remark}
We may obtain a banded unlink diagram for $U_\Sigma$ as follows.  Let $(\K_0,L,v)$ be a banded unlink diagram for $\Sigma\subset S^4$. Add a $1$-framed $2$-handle to $\K_0$ as a small meridian of $L$, far away from $v$, to obtain a Kirby diagram $\K_1$ for $\CP^2$. Then $(\K_1,L,v)$ is a banded unlink diagram for $U_\Sigma\subset\CP^2$.
\end{remark}

\subsection{Ribbon and $0$-concordant surfaces}\label{subsec:ribbon} We now proceed to define ribbon surfaces, as well as a diagrammatic framework for describing these surfaces and their isotopies, called \emph{chord diagrams}.  We will use chord diagrams to show that all ribbon unit surfaces are unknotted in $\CP^2$.

\begin{definition}
Let $S_1, S_2$ be surfaces in $S^4$, and let $\pi_I:S^4 \times I \rightarrow I$ be the projection to the unit interval. We say that $S_1$ is \emph{ribbon-concordant} to $S_2$ if there exists a cobordism $f:\Sigma\times I\to  S^4\times I$ so that $f(\Sigma\times 0)=S_1\times 0$ and  $f(\Sigma\times 1)=S_2\times 1$ and so that $\pi_I \circ f$ has finitely many critical points with distinct critical values, all of which are of index-$0$ or -$1$.  We say that a surface $S$ is {\emph{ribbon}} if the unknotted sphere is ribbon-concordant to $S$.
\end{definition}

Kawauchi~\cite{kawauchi2015chord} gives an equivalent definition of ribbonness without the cobordism perspective, by using semi-unknotted punctured handlebodies.

\begin{definition}[e.g.~\cite{kawauchi2015chord}]
A genus-$g$ surface $R\subset X^4$ is \emph{ribbon} if $R$ bounds a a punctured 3-dimensional handlebody $V$ embedded in $X^4$ so that $\partial V = R \cup O$, where $O$ is an trivial unlink of unknotted spheres.
\end{definition}

From this definition comes  a diagrammatic description of ribbon surfaces. We recall the definition of a chord diagram of an oriented ribbon surface knot $R \subset S^4$.

\begin{definition}[\cite{kawauchi2015chord}]\label{chorddef}
A \emph{chord graph} for a ribbon-surface in $S^4$ consists of an oriented unlink $o$ of circles in $S^3$ and arcs $\alpha$ in $S^3$ with endpoints on $o$ and interiors disjoint from $o$. This graph indicates the same ribbon-surface as the banded unlink $(\K_0,o,v)$, where $v$ consists of pairs of dual bands attached along the $\alpha$ curves, as in Fig.~\ref{fig:ribbondiagram}. (Twisting this pair does not affect the resulting surface, as long as they describe an orientable surface.)  A \emph{chord diagram} is a planar diagram of a chord graph. (See~\cite{kawauchi2015chord} for details.)
\end{definition}

\begin{figure}
\includegraphics[width=.85\textwidth]{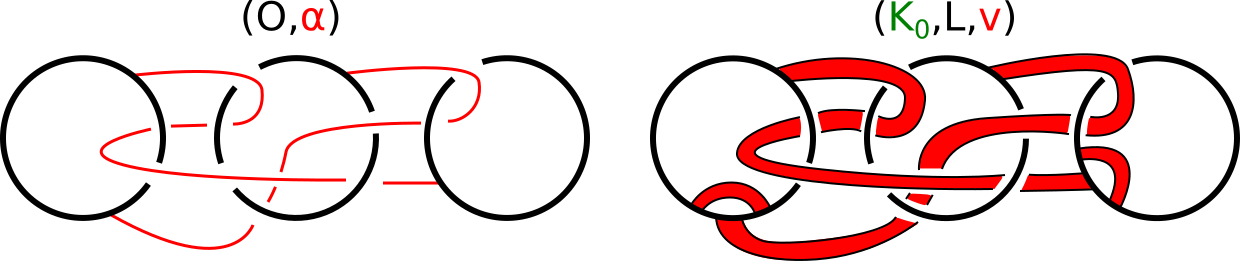}
\caption{Left: a chord diagram for a ribbon surface $R$ in $S^4$. Each unlink component is oriented counterclockwise. Right: a banded unlink diagram for the same ribbon surface $R$ in $S^4$.}\label{fig:ribbondiagram}
\end{figure}

Kawauchi~\cite{kawauchi2015chord} gives a list of diagrammatic moves on chord diagrams that represent isotopies of a ribbon surface $S^4$. (This list is incomplete, see e.g.~\cite{kawauchi2017supplement}.) These moves are illustrated in Figure~\ref{fig:m012moves}.


\begin{figure}
\includegraphics[width=\textwidth]{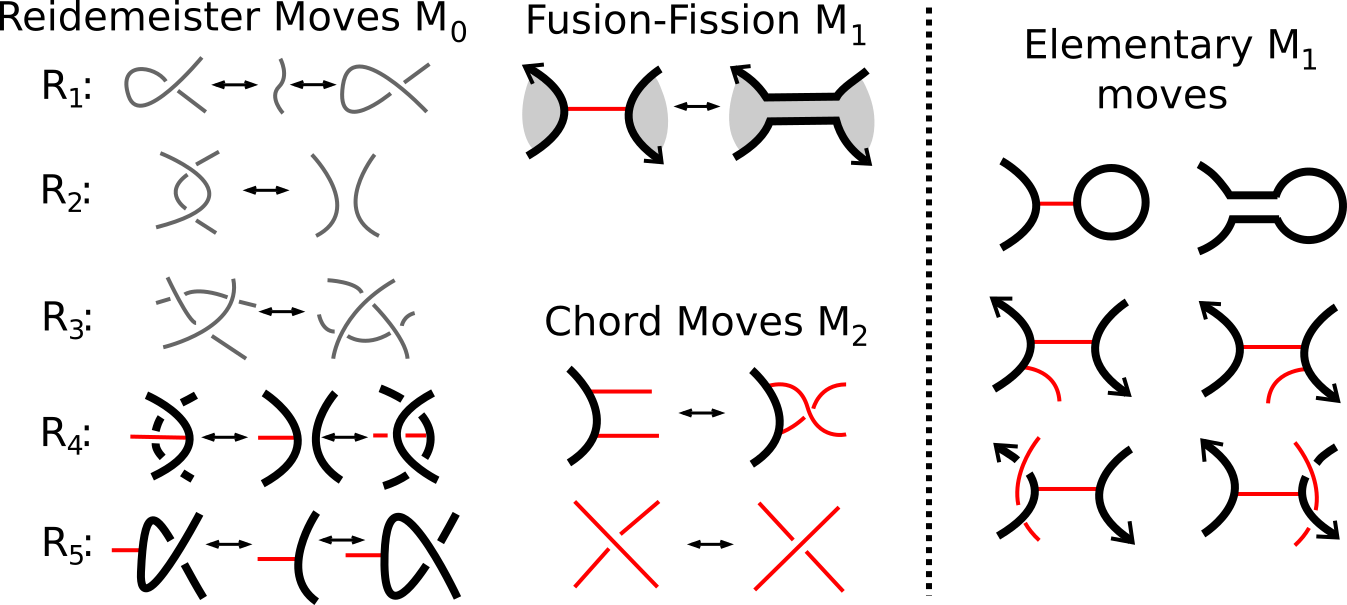}
\caption{Kawauchi's $M_0,M_1,M_2$ moves for chord diagrams in $S^4$. In moves $R_4, R_5, M_1$ and $M_2$ the red (narrow) curves represent $\alpha$ arcs, while black (bold) curves represent components of $o$.  The (gray) curves of moves $R_1, R_2,$ and $R_3$ represent arcs from either $\alpha$ or $o$.}\label{fig:m012moves}
\end{figure}


\begin{remark}
We do not specify the framings on the chords $\alpha$ in a chord diagram $(O,\alpha)$ for $R$. In fact, if two framings of $\alpha$ give descriptions of orientable surfaces, then the surfaces are isotopic. Similarly, we may add whole twists to the framing of $v$ in the associated banded unlink $(\K_0,L,v)$ for $R$ without changing the isotopy type of the described surface.
\end{remark}

\begin{exercise}
Let $(O,\alpha)$ be a chord diagram and $(\K_0,L,v)$ the associated banded unlink.
\begin{itemize}
\item An $M_0$ move on $(O,\alpha)$ induces isotopy on $(\K_0,L,v)$.
\item Each elementary $M_1$ move on $(O,\alpha)$ can be achieved by performing a sequence of fusion-fission $M_1$ moves and isotopies on $(O,\alpha)$. Conversely, every fusion-fission $M_1$ move can be achieved by a sequence of elementary $M_1$ moves and isotopies.
\item An elementary $M_1$ move on $(O,\alpha)$ induces isotopy and a sequence of band moves on $(\K_0,L,v)$.
\item An $M_2$ move on $(O,\alpha)$ induces isotopy and a sequence of band slides and swims on $(\K_0,L,v)$. (Very informally, the $M_2$ move most uses the property that $(O,\alpha)$ describes a ribbon surface. In a general banded unlink diagram, we cannot hope to pass bands through one another.)
\end{itemize}
\end{exercise}

When $R$ is ribbon, the moves of Theorem~\ref{thm:generalswenton} give a useful isotopy of $U_R$. We call this the $M_3$ move; see Figure~\ref{fig:m3move}. We slide a band over the $2$-handle, and then swim the $2$-handle through the dual band (or perform the same moves in the opposite order). This move allows us to add and remove linking between $L$ and $v$ in $E(\K_1)$.

\begin{figure}
\includegraphics[width=.9\textwidth]{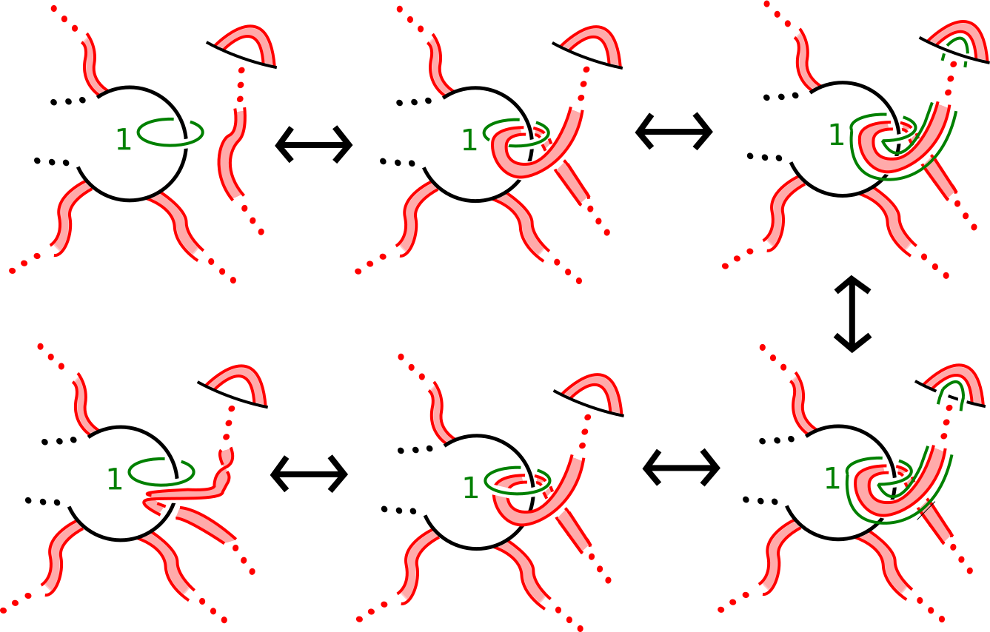}
\caption{The $M_3$ move consists of one $2$-handle-band slide and one $2$-handle-band swim. We slide a band over the $1$-framed $2$-handle, then swim the $2$-handle through the dual band (or perform the same moves in the opposite order). With this move, we can add or remove linking between $L$ and $v$ in $E(\K_1)$.}\label{fig:m3move}
\end{figure}

\begin{lemma}\label{ribbonlemma}
\label{lem:newmove}
Let $R\subset S^4$ be a ribbon surface of genus $g$. Then $U_R\subset\CP^2$ is isotopic to $\CP^1\# gT$, where $T$ is an unknotted torus in $B^4$.
\end{lemma}

\begin{proof}
Let $(\K_0,L,v)$ be a banded unlink diagram for $R$, as in Definition~\ref{chorddef}. Then $(\K_1,L,v)$ is a banded unlink diagram for $U_R$. Perform the $M_3$ move on $L\cup v$ in $\K_1$ finitely many times until $v$ does not link $L$. (That is, until the bands $v$ do not meet disks bounded by $L$ in $S^3\supset\K_1$.) We then do moves $M_0$ and $M_2$ finitely many times to trivialize $L$ and $v$ so that in the projection $\pi:E(K_1)\to\R^2$, $\pi(L\cup v)$ are embedded (recall we may choose the framing of $v$ to be the framing induced by the plane). Finally, we perform the fusion $M_1$ move finitely many times to get a banded unlink diagram with one circle and $g$ pairs of dual bands. This is a banded unlink diagram for $\CP^1$ trivially stabilized $g$ times, i.e. $\CP^1\# gT$. 
\end{proof}

Using the above argument, we actually prove a stronger fact about a more general class of knots. We recall a form of concordance introduced by Melvin~\cite{melvin1977thesis} for spheres, and extended to positive-genus surfaces by Sunukjian~\cite{sunukjian2013concordance}.

\begin{definition}[\cite{melvin1977thesis}]
Let $S_1$ and $S_2$ be genus-$g$ surfaces in $X^4$. We say that $S_1$ is $0$-concordant to $S_2$ if there exists a cobordism $f:\Sigma_g\times I \to  X^4\times I$ so that $f(\Sigma_g\times 0)=S_1\times 0$, $f(\Sigma_g\times 1)=S_2\times1$, $\pi_I\circ f$ has finitely many critical points mapped to distinct $t$'s, and away from critical values of $t$, $f^{-1}(X^4\times t)$ is a disjoint union of a genus-$g$ surface and some number of spheres.
\end{definition}

One can think of a $0$-concordance as a concordance consisting of $0$-handles, then $1$-handles each attached between distinct surface components, then $2$-handles which each create a new sphere component, and finally $3$-handles.

Sunukjian~\cite{sunukjian2013concordance} noted the following relation between $0$-concordance and ribbon-concordance. 

\begin{lemma}[\cite{sunukjian2013concordance}, Lemma 8.1]\label{sunukjianlemma}
Let $S_1$ and $S_2\subset X^4$ be genus-$g$ surfaces so that $S_1$ is $0$-concordant to $S_2$. Then there exists a genus-$g$ surface $S$ so that $S_1$ and $S_2$ are both ribbon-concordant to $S$.
\end{lemma}
\begin{proof}
Let $f:\Sigma_g\times I\to X^4$ be a $0$-concordance from $S_1$ to $S_2$. Isotope $f$ so that critical points occur in order of index. Let $t\in I$ so that all index-$0$ and -$1$ critical points of $f$ are contained in $\Sigma_g\times[0,t)$, and all index-$2$ and -$3$ critical points of $f$ are contained in $\Sigma_g\times(t,1]$. Let $S=f(\Sigma_g\times t)\subset X^4\times t=X^4$, where $X^4\times t$ is identified with $X^4$ via projection.

Let $g_1,g_2:\Sigma_g\times I\to X^4\times I$ be given by $g_1(x,s)=f(x,s/t)$ and $g_2(x,s)=f(x,(1-s)/(1-t))$. Then $g_i$ is a ribbon-concordance from $S_i$ to $S$.
\end{proof}

\begin{theorem}\label{thm:0concordant}

Let $S,S'\subset S^4$ be genus-$g$ surfaces so that $S$ is $0$-concordant to $S'$. Then $S\# \CP^1$ is isotopic to $S'\#\CP^1$ in $\CP^2$. In particular, if $S$ is $0$-concordant to the unknot, then $S\#\CP^1$ is isotopic to $\CP^1\# gT$.
\end{theorem}

\begin{proof}
By Lemma~\ref{sunukjianlemma}, there exists a surface $S''$ so that $S$ and $S'$ are each ribbon-concordant to $S''$. It is sufficient to prove the following statement.
\begin{proposition}
Suppose $S$ is ribbon-concordant to $S'$ via a ribbon-concordance consisting of $k$ index-$0$ critical points and $k$ index-$1$ critical points. Then $S\#\CP^1$ is isotopic to $S'\#\CP^1$ in $\CP^2$.
\end{proposition}

\begin{proof}
The above setup is equivalent to saying that $S'$ is given by tubing $S$ to an unlink $\sqcup_k O$ of $k$ unknotted spheres along $k$ narrow tubes around arcs $b$. (There is an extra restriction on the endpoints of $b$, as this tubing must yield a connected surface.) 
Via move $M_3$ (Fig.~\ref{fig:m3move}), in $\CP^2$ we may remove all intersections of $b$ with the balls bounded by $\sqcup_k O$. That is, $S'\#\CP^1$ is isotopic to $(S\#O\#\cdots\# O)\# \CP^1=S\#\CP^1$.
\end{proof}
This completes the proof of Theorem~\ref{thm:0concordant}.\end{proof}

The proof of Lemma~\ref{ribbonlemma} also yields a result about band-sums. Recall the definition of band-summing.

\begin{definition}
Let $S_1$ and $S_2$ be oriented surfaces in $S^4$, contained in disjoint balls. Let $\gamma$ be an arc from a point on $S_1$ to a point on $S_2$. Then the band-sum $S_1\#_\gamma S_2$ is the surface $((S_1\sqcup S_2)\setminus\nu(\gamma))\cup(\gamma\times S^1)$.
\end{definition}
The framing on $\gamma$ to determine the $S^1$-bundle over $\gamma$ does not affect the resulting surface up to isotopy, so long as $\nu(\gamma)\cap(S_1\#_\gamma S_2)$ is oriented consistently with the orientations on $S_1$ and $S_2$.

Note that connect-summing is a specific example of band-summing. Now we show that blowing up $S^4$ trivializes the band-sum.

\begin{theorem}\label{thm:bandsum}
Let $S_1\#_\gamma S_2$ be a band-sum in $S^4$, where $S_1$ and $S_2$ are any smooth surfaces in $S^4$ contained in disjoint $4$-balls and $\gamma$ is any path between them. Then the unit surfaces $S_1\#_\gamma S_2\#\CP^1$ and $S_1\# S_2\#\CP^1$ are isotopic in $\CP^2$.
\end{theorem}

\begin{proof}
Fix band diagrams for $S_1$ and $S_2$. Then $S_1\#_\gamma S_2\#\CP^1$ has a banded unlink diagram $(\K_1,L,v)$ consisting of the union of the banded unlink diagrams for $S_1,S_2$ with two bands for the band-sum tube and the $+1$-surgery curve of $\CP^2$. The $M_3$ move allows one to change any crossing of $L$ with $\gamma$. (The $M_2$ move allows one to change any crossing of $v$ with $\gamma$ and to slide the endpoints of $\gamma$ along $L$, through $v$.) Therefore, in $\CP^2$ we may take $\gamma$ to be an arbitrary arc, and in particular find $S_1\#_\gamma S_2\#\CP^1$ is isotopic to $S_1\# S_2\#\CP^1$.
\end{proof}

\subsection{Deform-spun knots}\label{subsec:twist}
We move on from $0$-concordant knots to twist-spun knots~\cite{zeeman1965twisting} and deform-spun knots~\cite{litherland1979deforming}.

\begin{definition}[\cite{litherland1979deforming} generally,~\cite{zeeman1965twisting} for twist-spun knots]
Let $K^1\subset B^3$ be a $1$-stranded tangle. Let $f:B^3\to B^3$ be a diffeomorphism fixing $\boundary B^3$ and $K^1$ pointwise. Then the {\emph{$f$-deform spun knot}} of $K^1$ is $f K^1=K^1\times I/\sim$, contained in $S^4=(B^3\times I)/((x,1)\sim(f(x),0), y\times S^1\sim\pt$ for $y\in\boundary B^3$).
Let $T \subset B^3$ be the torus $\boundary(\nu(K^1\cup \boundary B^3))$. Parameterize $T=S^1\times S^1$ so that $S^1\times 0$ is a meridian of $K^1$ and $[0\times S^1]=0\in H_1(B^3\setminus K^1)$ (i.e. $0\times S^1$ is a $0$-framed longitude for $K^1$). Let $T\times I$ be a regular neighborhood of $T$ contained in the interior of $B^3\setminus K^1$.

Define $\tau,\rho:T\times I$ by $\tau(x,y,t)=(x+2\pi t,y,t)$, $\rho(x,y,t)=(x,y+2\pi t,t)$. Extend $\tau,\rho$ to the rest of $B^3$ by the identity map. Then $\tau^n \rho^p K^1$ is called the {\emph{$n$-twist $p$-roll spun knot}} of $K^1$. (When $n=0$ or $p=0$, we may say $p$-roll spun knot of $K^1$ or $n$-twist spun knot of $K^1$, respectively.) See Figure~\ref{fig:twistball} if this construction is unfamiliar.

\begin{figure}
\includegraphics[width=.35\textwidth]{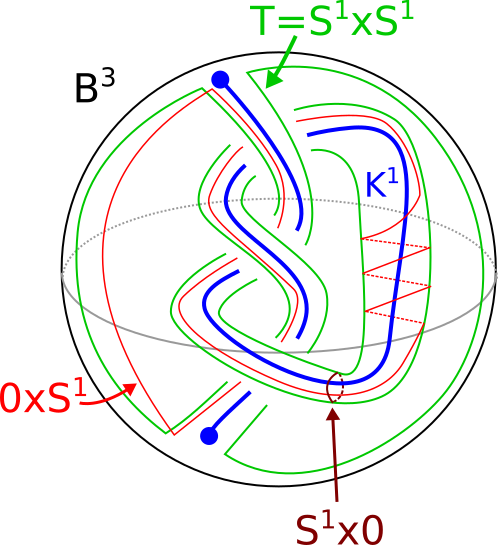}
\caption{The setup for Litherland's description of twist-roll spun knots~\cite{litherland1979deforming}. Here, $K^1$ is a $1$-stranded tangle in ball $B^3$. The torus $T$ is given by $\boundary(\nu(K^1\cup\boundary B^3))$. We parameterize $T\cong S^1\times S^1$ so that its longitude $0\times S^1$ is nullhomologous in $B^3\setminus\nu(K^1)$, and its meridian $S^1\times 0$ bounds a disk in $\nu(K^1)$.}\label{fig:twistball}
\end{figure}

Let $K$ be a classical knot in $S^3$ (that is, a {\emph{$1$-knot}}) and $B\subset S^3$ a small $3$-ball meeting $K$ in a trivial arc. If $(B^3, K^1)=(S^3\setminus B, K\setminus B)$, then we may write $f K$ to indicate $f K^1$, and refer to the $f$-deform spun knot of $K^1$ and the $f$-deform spun knot of $K$ interchangeably.
\end{definition}

\begin{figure}
\includegraphics[width=\textwidth]{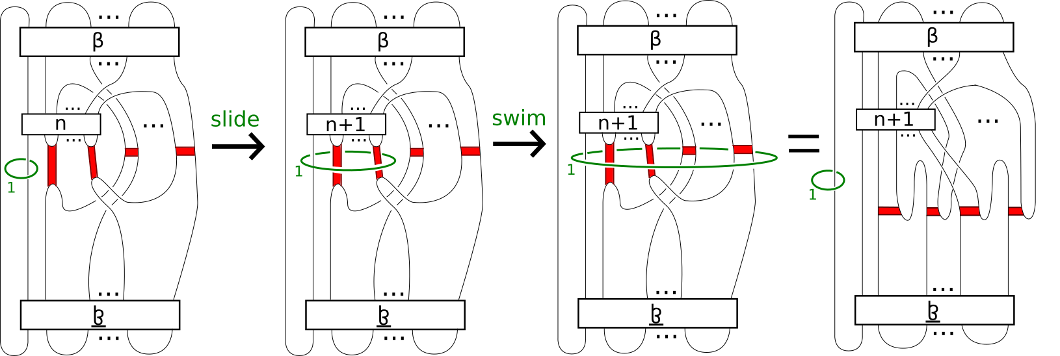}
\caption{Left to right: a banded unlink diagram for $\tau^n K\#\CP^1$ (see e.g.~\cite[Section 5.2]{meier2017bridge}), where $K$ in bridge position is the closure of tangle $\beta$. We slide half the bands over the $2$-handle of $\CP^2$, and then swim the $2$-handle through the remaining bands. We obtain a banded unlink for $\tau^{n+1} K\#\CP^1$.}\label{fig:twistspun}
\end{figure}

\begin{theorem}\label{thm:twistspun}
Let $K$ be a $1$-knot, so that $\tau^n K$ is the $n$-twist spun knot of $K$. 
Then $U_{\tau^n K}$ is isotopic to $\CP^1$.
\end{theorem}
\begin{proof}
In Figure~\ref{fig:twistspun}, we demonstrate an isotopy in $\CP^2$ taking $U_{\tau^n K}$ to $U_{\tau^{n+1} K}$.

 Inductively, $U_{\tau^n K}$ is isotopic to $U_{\tau K}$. By~\cite{zeeman1965twisting}, $\tau K$ is the unknot, so $U_{\tau^n K}$ is isotopic to $\CP^1$.
\end{proof}

We can say something stronger about the general family of deform-spun knots.

\begin{theorem}\label{thm:twistdeform}
Let $K$ be a $1$-knot. Let $f K$ be a deform-spun knot of $K$. Then $f K\#\CP^1$ is isotopic to $\tau^n f K\#\CP^1$ for any $n$.
\end{theorem}

\begin{proof}
Figure~\ref{fig:twistdeform} shows an explicit isotopy from $f K\#\CP^1$ to $\tau f K\#\CP^1$. 
\end{proof}

\begin{figure}
\includegraphics[width=\textwidth]{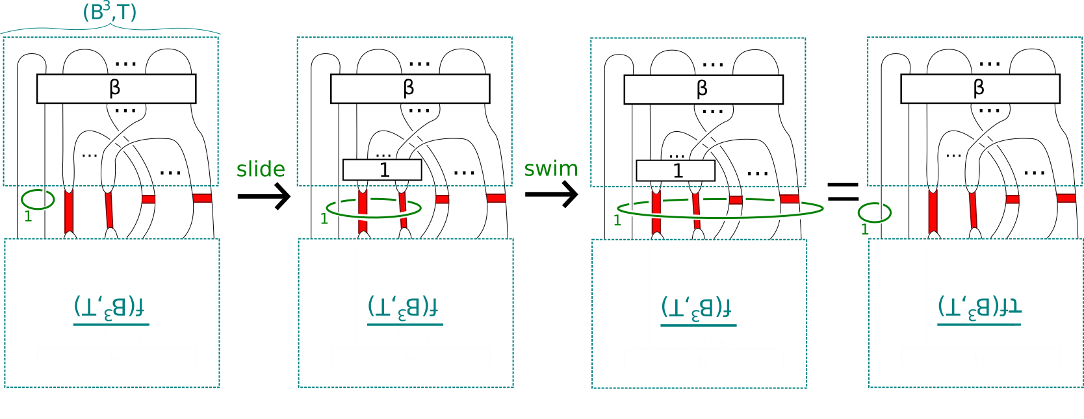}
\caption{Left to right: a diagram for $fK\#\CP^1$, where $K$ in bridge position is the closure of tangle $\beta$. We slide half the bands over the $2$-handle of $\CP^2$, and then swim the $2$-handle over the remaining bands. We obtain a diagram for $\tau fK\#\CP^1$.}\label{fig:twistdeform}
\end{figure}

\begin{corollary}\label{bergecor}
Let $K$ be a knot with an integral lens space surgery. Then $\tau^n \rho K\#\CP^1$ is isotopic to the standard $\CP^1$ for any $n\in \Z$.
\end{corollary}

\begin{proof}
Say $k$-surgery on $K$ yields a lens space $L(k,q)$. Teragaito~\cite{teragaito1994twistroll} observed that $\tau^k \rho K$ is the unknot, as follows:

By Litherland~\cite{litherland1979deforming}, $\tau^k \rho K$ is a fibered knot, with closed fiber obtained by $1$-Dehn filling the $k$-fold cyclic cover of $S^3\setminus\nu(K)$. That is, the closed fiber is a $k$-fold cover of $L(k,q)$, so the closed fiber is $S^3$. Therefore, $\tau^k \rho K$ bounds a smooth $3$-ball.

By Theorem~\ref{thm:twistdeform}, $\tau^n \rho K\#\CP^1$ is isotopic to $\tau^k \rho K\#\CP^1=\CP^1$ for any $n$.
\end{proof}

\subsection{Satellites and miscellaneous examples}
Consider the family of $2$-knots $K_{pq}$ illustrated in Figure~\ref{fig:nash} (top). Nash and Stipsicz~\cite{nash2012gluck} showed via Kirby calculus that the Gluck twist on any of these $2$-knots yields $S^4$. In fact, by translating their handle slides into band moves, we observe that $K_{pq}\#\CP^1$ is isotopic to the standard $\CP^1$ in $\CP^2$.

\begin{figure}
\includegraphics[width=.75\textwidth]{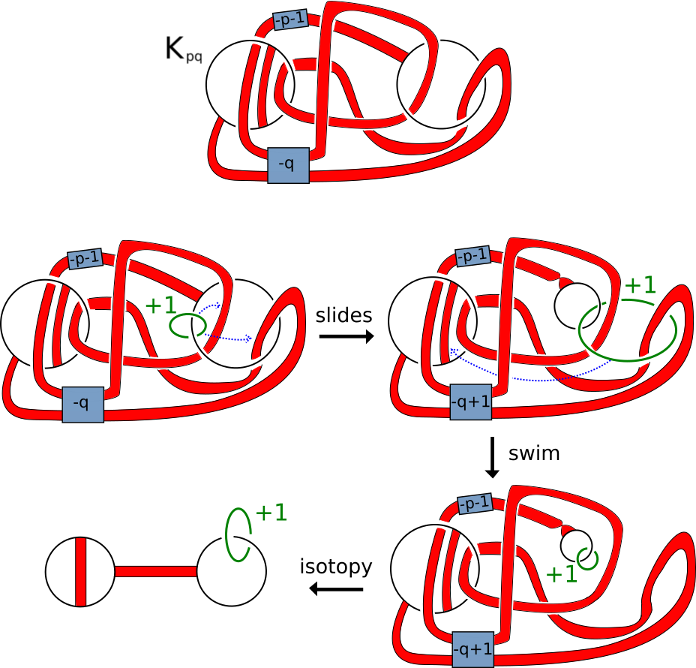}
\caption{Top: The $2$-knot $K_{pq}\subset S^4$. Nash and Stipsicz~\cite{nash2012gluck} showed that performing the Gluck twist on $K_{pq}$ yields $S^4$. Bottom two rows: An isotopy from $K_{pq}\#\CP^1$ to the standard $\CP^1$ in $\CP^2$.}\label{fig:nash}
\end{figure}

Most of the results of Subsections~\ref{subsec:ribbon} and~\ref{subsec:twist} can be consolidated into the single following statement.

\begin{theorem}\label{cp2theorem}
Let $F=S\#\CP^1\subset\CP^2$ be a genus-$g$ unit surface knot, where $S\subset S^4$ is an orientable surface that is $0$-concordant to a band-sum of twist-spun knots and unknotted surfaces. Then $F$ is isotopic to $\CP^1\#gT$, where $\CP^1\# gT$ indicates the standard $\CP^1$ trivially stabilized $g$ times.
\end{theorem}

The results of Theorems~\ref{thm:twistspun} and~\ref{thm:twistdeform} extend to satellite knots. This illustrates the strength of this diagrammatic packaging, as in general these knots may not be twist-spins or even fibered (see e.g.~\cite{yoshikawa1982satellite}).

\begin{definition}
Let $K_P\subset V$ be a $2$-sphere embedded in $V\cong S^2\times D^2$. Let $K_C$ be a $2$-sphere embedded in $S^4$. Fix a diffeomorphism $\phi:V\to \nu(K_C)$. Let $K=f(K_P)\subset S^4$. We call $K$ the {\emph{satellite}} of {\emph{companion}} $K_C$ with {\emph{pattern}} $(K_P,V)$.
\end{definition}

Let $K$ be the satellite of companion $K_C$ with pattern $(K_P,V)$. To obtain a diagram of $K$, we view $S^4=S^3\times[-1,1]/(S^3\times 1\sim\pt, S^3\times -1\sim\pt)$, where $K_C\cap S^3\times 0$ is a connected knot and $K_C\cap (S^3\times[-1,0])$ and $K_C\cap (S^3\times[0,1])$ are ribbon disks (this is the normal form of~\cite{kawauchi1982normalform}).

 We take $V=S^2\times D^2\subset S^4$ so that $W:=V\cap (S^3\times 0)\cong S^1\times D^2$, and $V\cap(S^3\times[-1,0])\cong V\cap (S^3\times[0,1])\cong D^2\times D^2$. See Figure~\ref{fig:pattern} (left) for a schematic.

\begin{figure}
\includegraphics[height=1.5in]{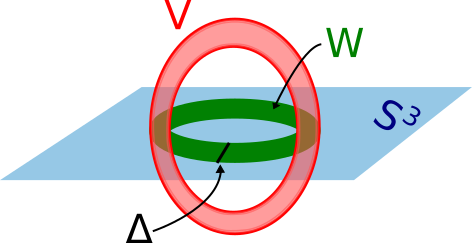}\hspace{.25in}\includegraphics[height=1.5in]{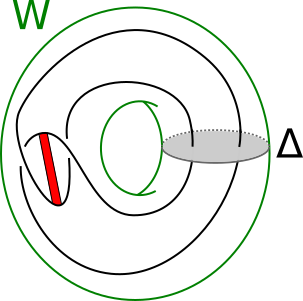}
\caption{{\bf{Left:}} A schematic of $V$ embedded in $S^4$. The intersection $V\cap S^3\times 0$ is the solid torus $W$. {\bf{Right:}} A banded unlink of a pattern $K_P\subset V$. With respect to this choice of $\Delta$, the pattern has geometric winding $2$. (Note in this diagram, $K_P$ is unknotted inside $V$, so could be isotoped to a pattern of geometric winding $0$.)}\label{fig:pattern}
\end{figure}

Draw a banded unlink (in $\K_0$) for $K_P$ sitting inside $W$. (Note the original unlink and the one obtained by resolving all bands are unlinked in $B^3=W\cup(0$-framed $2$-handle$)$ but may be nontrivial within the solid torus $W$.) Fix a meridian disk $\Delta$ of $W$ disjoint from all bands in the diagram for $K_P$, and take $K_P$ to be transverse to $\Delta$. We say the {\emph{geometric winding}} of $K_P$ is the unsigned intersection $K_P\cap\Delta$. (Note this number depends on $\Delta$.)

Now draw a banded unlink (in $\K_0$) for $K_C$, isotoped to lie inside $W':=\nu(K_C)\cap(S^3\times 0)\cong S^1\times D^2$ (we first draw the diagram in $S^3=S^3\times 0 $ and perturb to be disjoint from a core circle of $S^3\setminus W$, and then project to $W$.) Fix a meridian disk $\Delta'$ for $W'$ which intersects the banded unlink transversely in one point.

Isotope the diffeomorphism $\phi:V\to\nu(K_C)$ so that $\phi(W)=W'$, and $\phi(\Delta\times I)=W'\setminus(\Delta'\times I)$. Choose $\phi$ so that every saddle of $K$ either corresponds to a saddle of $K_P$ or lies above or below a saddle of $K_C$. Each saddle of $K_C$ gives rise to $n$ saddles of $K$. 
%

Therefore, the satellite $K$ has $a+nb$ critical points of index $1$, where $a$ is the number of index-$1$ critical points of $K_P$ and $b$ the number of index-$1$ critical points of $K_C$. Then we obtain a banded unlink diagram (in $\K_0$) for $K$ by taking the standard $0$-framed satellite of $K_P\cap (S^3\times 0)\subset W$ around $K_C\cap(S^3\times 0)$ (both of these cross-sections are knots), attaching the bandsbands corresponding to $K_P$, and then attaching $n$ copies of each bands corresponding to $K_C$ (pushing them above $S^3\times 0$). See Figure~\ref{fig:satellite}.

\begin{figure}
\includegraphics[width=\textwidth]{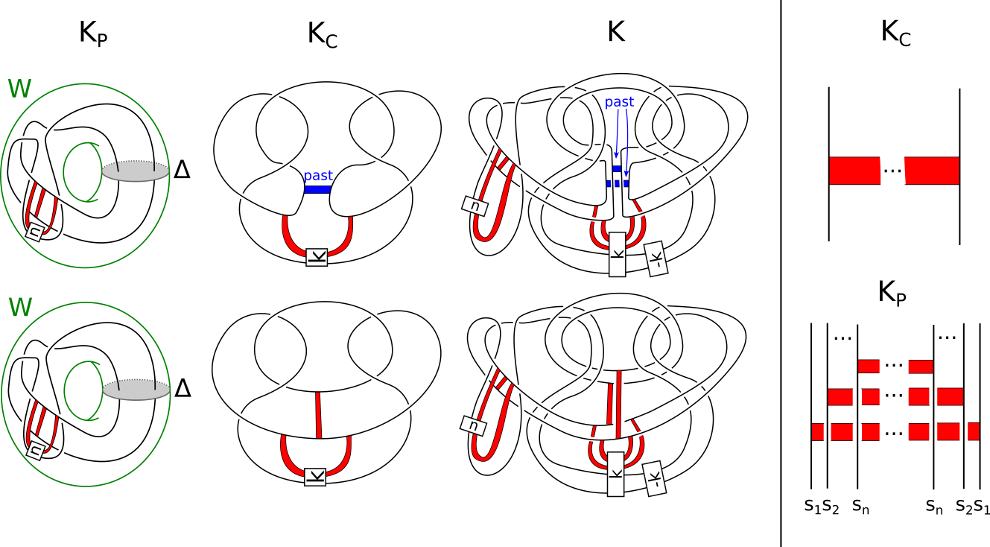}
\caption{{\bf{Left, top row:}} We draw a banded unlink diagram for $K_P$ and a normal form diagram for $K_C$. (a band diagram in which some bands lie below the pictured cross-section, as indicated, so that intersection of $K_C$ with this time slice is a knot). From these two diagrams, we obtain a diagram of $K$. {\bf{Left, bottom row:}} We push all bands above the pictured cross-section, so that each diagram is a banded unlink diagram. {\bf{Right:}} We show how to draw the $n$ bands of a banded unlink diagram for $K$ corresponding to a band in the banded unlink diagram for $K_C$. In this picture, $s_i$ indicates one arc in $\phi(\Delta\times I)$.}
\label{fig:satellite}. 
\end{figure}

We give a new isotopy move of $\CP^1$ inside $(S^2\times D^2)\#\CP^2$: the double slide.  Refer to Figure~\ref{fig:doubleslide}. 
The effect of the double slide move is to change the intersection of $\CP^1$ with $W$ by two slides over a longitude of $W$ (with opposite sign). Via this move, we may replace the intersection of $\CP^1$ with $W$ with any isotopic curve in $W\cup 0$-framed $2$-handle $\supset W$ representing the same element of $H_1(W)$.

\begin{figure}
\includegraphics[width=\textwidth]{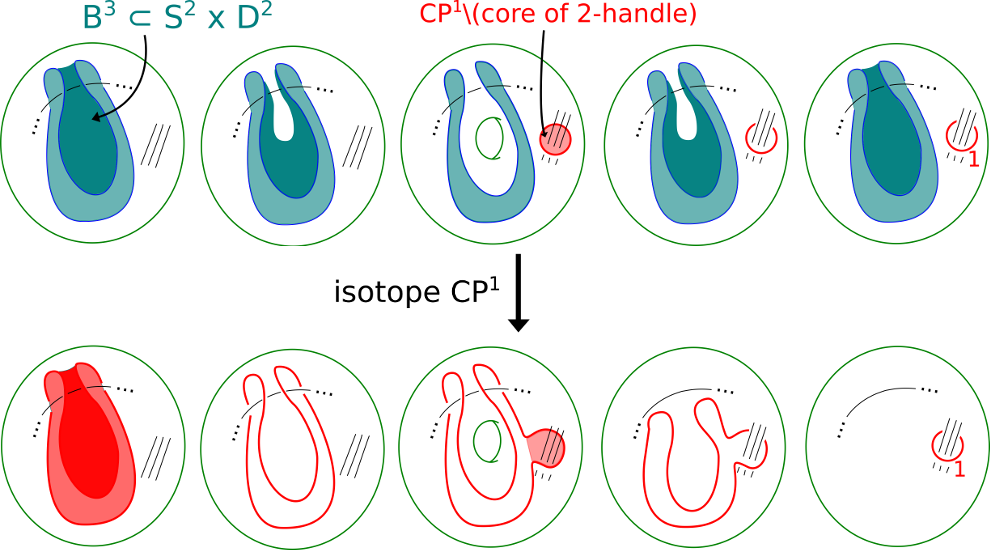}
\caption{{\bf{Top:}} a movie of cross-sections of $S^2\times D^2\#\CP^2$. The central cross-section is a solid torus $W$. In shaded blue, we draw a ball $B$ in $S^2\times D^2$. In red (to the right), we draw the disk $\CP^1\cap S^2\times D^2$. The rest of $\CP^1$ is a core of the $1$-framed $2$-handle. Black strands are contained in $K_P\subset S^2\times D^2$. {\bf{Bottom:}} We isotope $\CP^1$ through $B$ while fixing the unit surface for $K_P$. }\label{fig:doubleslide}
\end{figure}

\begin{theorem}\label{satellitetheorem}
Let $K$ be a satellite of companion $K_C$ with pattern $(K_P,V)$. Say $[K_P]=m[S^2\times \pt]$ in $H_2(V;\Z)$. View $V\subset S^4$ as a neighborhood of an unknotted $2$-sphere. Assume $K_P\subset V\subset S^4$ is $0$-concordant to a band-sum of twist-spun knots.

\begin{itemize}
\item If $m=0$, then $K\#\CP^1$ is isotopic to $\CP^1$ in $\CP^2$.
\item If $m=\pm 1$, then $K\#\CP^1$ is isotopic to $\pm K_C\#\CP^1$ in $\CP^2$.
\item If $|m|>1$, then $K\#\CP^1$ is isotopic to $K'\#\CP^1$, where $K'$ is a satellite with companion $K_C$ and pattern $(O_m, V)$, where $O_m\subset V\subset S^4$ is an unknotted sphere and $[O_m]=m[S^2\times\pt]\in H_2(V;\Z)$.
\end{itemize}

\end{theorem}

\begin{remark}
The pattern $O_m\subset V$ is well-defined. We have $V= S^4\setminus \nu(\gamma)$. where $\gamma$ is a curve in $S^4\setminus U_m$ with $[\gamma]=m\subset\Z\cong\pi_1(S^4\setminus U_m)$. In this dimension, $\gamma$ is unique up to isotopy, so $U_m\subset V$ is uniquely determined by $m$. See Figure~\ref{fig:mpattern} for a banded unlink diagram (with no bands) for $O_m\subset V$. Note that the satellite knot $K'$ is isotopic to $|m|$ parallel copies of $(m/|m|)K_C$ tubed together.
\end{remark}

\begin{figure}
\includegraphics[width=.75\textwidth]{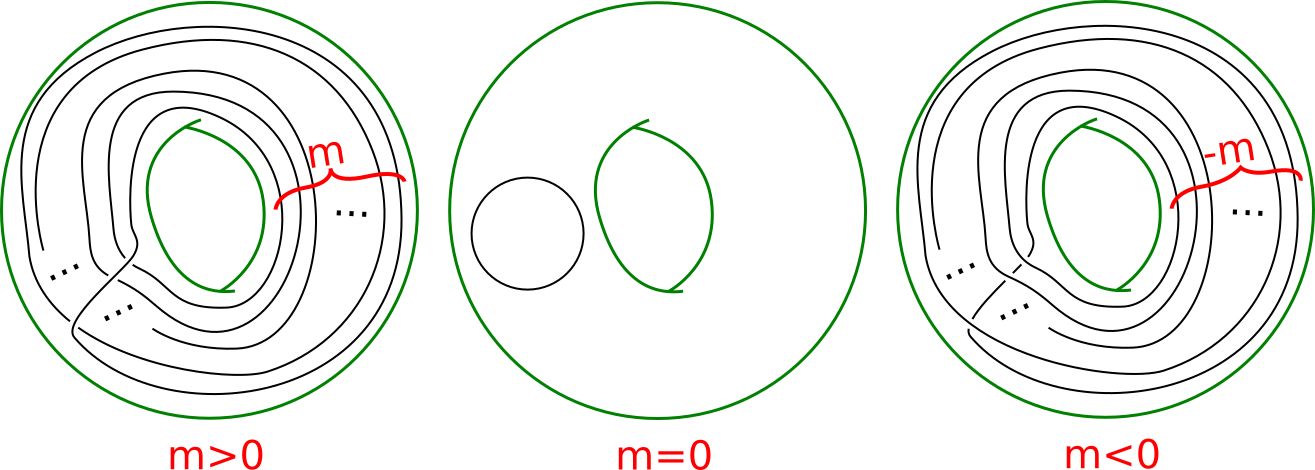}
\caption{A banded unlink diagram for $O_m\subset V$. In this diagram, $O_m$ has one local minima, one local maxima, and zero saddles (i.e. zero bands).}\label{fig:mpattern}. 
\end{figure}

%

\begin{proof}[Proof of Theorem~\ref{satellitetheorem}]
Let $F$ be the surface $K_P\#\CP^1\subset (S^2\times D^2)\#\CP^2$. The banded unlink diagram for $K_P\#\CP^1\subset\CP^2$ sits inside the solid torus $W\subset S^3$, with a $1$-framed $2$-handle attaching circle $\gamma$ at the site of the blowup. See Figure~\ref{fig:unknotpattern}. Apply the isotopy of Theorem~\ref{cp2theorem} to unknot the banded unlink diagram. So long as $\gamma$ stays in $W$, this isotopy induces isotopy on $F$ in $(S^2\times D^2)\#\CP^2$. But $\gamma$ passes outside of $W$ (``through the hole of $W$'') an even number of times -- that is, $\gamma$ appears to slide over a longitude of $W$ an even number of times; an equal number of each direction of slide. Achieve these slides through a sequence of double slide moves. See Figure~\ref{fig:unknotpattern}. This isotopy does not fix the standard $\CP^1$, but replaces $F$ with $U_m\#($standard $\CP^1)$.

Let $K'\subset S^4$ be the satellite knot with pattern $O_m$ and companion $K_C$. Then $K\#\CP^1$ is isotopic to $K'\#\CP^1$ in $\CP^2$. Note that if $m=0$, then $K'$ is the unknot. If $m=\pm 1$, then $K'=\pm K_C$.
\end{proof}

\begin{figure}
\includegraphics[width=.75\textwidth]{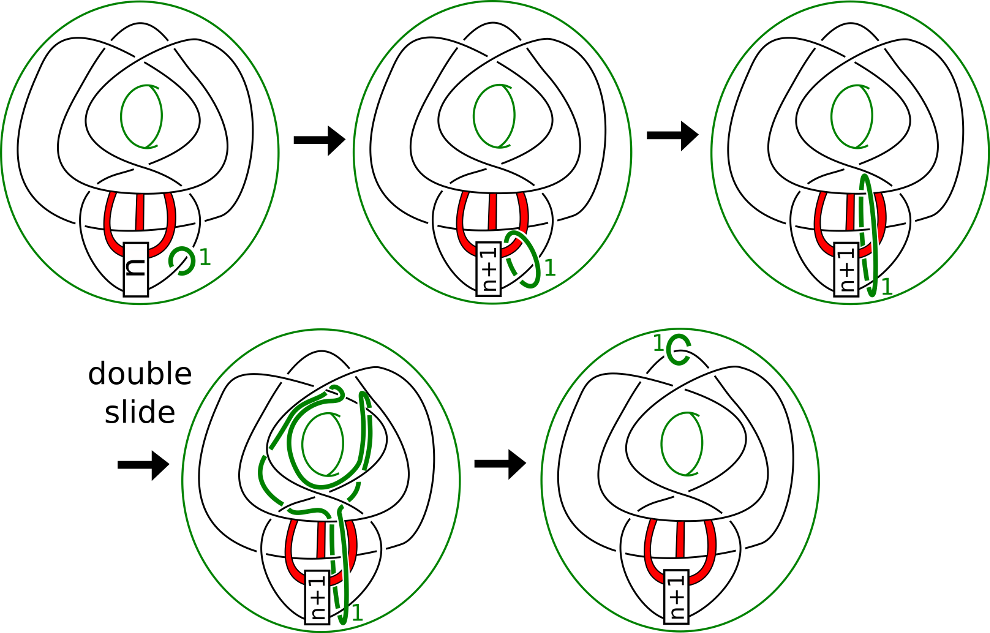}
\caption{We unknot $K_P\#\CP^1$ in $S^2\times D^2\#\CP^2$ when $K_P$ is $0$-concordant to a band-sum of twist-spun knots. We perform the isotopy of Theorem~\ref{cp2theorem}, performing the double slide move to slide $\CP^1\cap W$ over the longitude of $W$ and back.}\label{fig:unknotpattern}. 
\end{figure}

The following corollary follows immediately from Theorem~\ref{satellitetheorem} and the relation of unit surfaces to the Gluck twist from~\cite{melvin1977thesis}.
\begin{corollary}
Let $K$ be a satellite with companion $K_C$ and pattern $(K_P,V)$. Say $[K_p]=m[S^2\times \pt]$ in $H_2(V;\Z)$. View $V\subset S^4$ as a neighborhood of an unknotted $2$-sphere. Assume $K_P\subset V\subset S^4$ is $0$-concordant to a band-sum of twist-spun knots.

\begin{itemize}
\item If $m=0$, then the Gluck twist on $S^4$ about $K$ is diffeomorphic to $S^4$.
\item If $m=\pm 1$ and the Gluck twist on $S^4$ about $K_C$ yields $S^4$, then so does the Gluck twist on $S^4$ about $K$.
\end{itemize}
\end{corollary}
\label{endsec5}

\bibliographystyle{plain}
\bibliography{bibliography.bib}

\end{document}